\documentclass[smallextended]{svjour3}       
\smartqed  
\usepackage{graphicx}
\usepackage{amsfonts}
\usepackage{amsmath}

\usepackage{lipsum}
\usepackage{epstopdf}
\usepackage{algorithmic}
\usepackage{url}
\usepackage{color}
\usepackage{todonotes}

\newcommand{\bfE}{{\mathbf{E}}}
\newcommand{\bfH}{{\mathbf{H}}}

\newcommand{\N}{\mathbb{N}}

\newcommand{\bfU}{\mathbf U}

\begin{document}

\title{Modal analysis of a domain decomposition method for Maxwell's equations in a waveguide}

\titlerunning{Modal analysis of a DDM for Maxwell}        %

\author{V. Dolean  \and A. Tonnoir \and P.-H. Tournier  
}

\institute{V. Dolean \at
              Eindhoven University of Technology, PO Box 513, Eindhoven, 5600 MB, The Netherlands \\
              \email{v.dolean.maini@tue.nl}       
 \and 
           A. Tonnoir \at 
              INSA Rouen Normandie, Normandie Univ, LMI UR 3226, F-76000 Rouen, France \\
              \email{antoine.tonnoir@insa-rouen.fr}
           \and
           P.-H. Tournier \at
              Sorbonne Université, Université Paris Cité, CNRS, Inria, Laboratoire Jacques-Louis Lions, LJLL, EPC ALPINES, 4 place Jussieu, Paris F-75005, France \\
              \email{pierre-henri.tournier@sorbonne-universite.fr}      
}

\date{Received: date / Accepted: date}

\maketitle

\begin{abstract}
Time-harmonic wave propagation problems, especially those governed by Maxwell’s equations, pose significant computational challenges due to the non-self-adjoint nature of the operators and the large, non-Hermitian linear systems resulting from discretization. Domain decomposition methods, particularly one-level Schwarz methods, offer a promising framework to tackle these challenges, with recent advancements showing the potential for weak scalability under certain conditions. In this paper, we analyze the weak scalability of one-level Schwarz methods for Maxwell’s equations in strip-wise domain decompositions, focusing on waveguides with general cross sections and different types of transmission conditions such as impedance or perfectly matched layers (PMLs). By combining techniques from the limiting spectrum analysis of Toeplitz matrices and the modal decomposition of Maxwell’s solutions, we provide a novel theoretical framework that extends previous work to more complex geometries and transmission conditions. Numerical experiments confirm that the limiting spectrum effectively predicts practical behavior even with a modest number of subdomains. Furthermore, we demonstrate that the one-level Schwarz method can achieve robustness with respect to the wave number under specific domain decomposition parameters, offering new insights into its applicability for large-scale electromagnetic wave problems.
\keywords{Maxwell’s equations \and Schwarz methods \and Domain decomposition \and Weak scalability \and Waveguide problems \and Limiting spectrum \and Block Toeplitz matrices \and modal decomposition.}

\end{abstract}

\maketitle

\section{Introduction}
Time-harmonic wave problems, particularly those that arise in electromagnetic applications governed by Maxwell’s equations, present significant computational challenges. At the continuous level, these problems involve non-self-adjoint operators when impedance boundary conditions are imposed. When discretized, the number of degrees of freedom must grow with the wave number to mitigate the pollution effect, which means that the numerical wave speed deviates from the exact solution \cite{Babuska:1997:IPE}. This increase in discretization leads to large-scale, non-Hermitian linear systems that are difficult to solve using traditional iterative methods.

During the past two decades, significant advances have been made in the development of efficient solvers and preconditioners for these problems. Among these approaches, domain decomposition methods \cite{Dolean:15:DDM} offer an effective balance between direct and iterative strategies. Enhanced domain decomposition techniques, such as those using optimized transmission conditions, have proven successful for Helmholtz equations \cite{Gander:2007:OSM} and their extensions to Maxwell’s equations \cite{Dolean:2015:ETC,Dolean:09:OSM,ElBouajaji:12:OSM} and elastic wave problems \cite{Brunet:2019:NDD,Mattesi:2019:ABC}. For large-scale problems, robustness in terms of subdomain count and wave number has been achieved by introducing two-level methods that leverage absorptive counterparts of the equations as preconditioners, solved iteratively using domain decomposition techniques \cite{Bonazzoli:2019:ADD,Dolean:2020:IFD,Graham:2017:RRD}.

More recently, an intriguing concept has emerged: achieving weak scalability with one-level Schwarz methods under certain conditions on the problem’s physical and numerical parameters, such as absorption and subdomain size. This ensures that the convergence rate remains stable as the number of subdomains increases, enabling {the computation of} the solution of increasingly complex problems without requiring a coarse space \cite{Gong:2020:DDP,Graham:2020:DDI}. Unlike traditional scalability that pertains to a fixed problem, weak scalability applies to a family of problems, where increasing the number of subdomains facilitates the solution of more challenging instances while maintaining consistent convergence rates. This concept has been explored in computational chemistry \cite{Cances:2013:DDI} and analyzed rigorously using Fourier techniques \cite{Ciaramella:2017:APS}. Extensions of this work to broader geometries and one-level methods have been achieved through variational and propagation-tracking analyses \cite{Ciaramella:2020:OSS}. Such analyses have been generalized only later to complex valued problems, decompositions into multiple subdomains and optimisation of transmission conditions in \cite{Dolean:2023:CFO,Dolean:2020:IFD,Gander:2024:ASM}. 

Notably in \cite{Bootland:2022:APS} authors investigates the convergence properties of one-level parallel Schwarz methods with Robin transmission conditions for time-harmonic wave problems, focusing on 1D and 2D Helmholtz and 2D Maxwell equations. By utilizing the block Toeplitz structure of the Schwarz iteration matrix, the authors provide a novel analysis of the limiting spectrum, showing that weak scalability can be achieved without a coarse space under specific conditions, particularly in strip-wise decompositions commonly found in waveguide problems. {Their analysis, however, was restricted to 1D and 2D Helmholtz and simplified 2D Maxwell settings with Robin conditions. 

Our work extends this framework to full Maxwell waveguides of arbitrary cross-section, where the modal decomposition into TE/TM/TEM modes is essential.} The main difficulty in extending these results to Maxwell’s equations lies in the vector nature of the problem and in the coupling between field components through the curl–curl operator. As a consequence, the Schwarz iteration cannot be analyzed directly using scalar techniques. The key idea of this work is to show that, in a waveguide geometry, the modal decomposition into TE, TM, and TEM modes diagonalizes not only the continuous Maxwell operator but also the interface transmission operators used in the Schwarz method (see also \cite{Dolean:2015:ETC}). This property allows us to reduce the analysis of the vector Maxwell problem to a family of independent scalar problems at the modal level. As a result, the Schwarz iteration matrix for Maxwell’s equations can be written in a block Toeplitz form identical to the one obtained for the Helmholtz equation, which makes it possible to apply limiting spectrum techniques to the full Maxwell system.

{The results of these paper can also be seen as the first analysis of weak scalability of one-level Schwarz methods for Maxwell’s equations in waveguides of general cross-section, using a combination of Toeplitz spectrum analysis and modal decomposition. Unlike earlier studies limited to simplified configuration, our framework covers general transmission conditions (impedance, PML) and fully accounts for the vector nature of Maxwell’s equations.}
While previous studies dedicated to Helmholtz problems focused on wave number robustness \cite{Gong:2020:DDP,Graham:2020:DDI}, our emphasis is on scalability over a growing chain of subdomains with fixed size, independent of discretisation. This approach provides new insights into the efficiency and applicability of domain decomposition methods for large-scale electromagnetic wave problems. The main contributions of this paper are:
\begin{itemize}
\item We show that, in waveguide geometries, the Schwarz method for Maxwell’s equations can be analyzed through a modal decomposition, which reduces the vector problem to a family of independent scalar problems.

\item We establish that the resulting Schwarz iteration has a block Toeplitz structure identical to that arising in the scalar Helmholtz case, revealing a precise correspondence between Maxwell and Helmholtz problems.

\item This correspondence allows us to extend limiting spectrum and weak scalability results, previously known for Helmholtz problems, to the full Maxwell system in waveguides of general cross-section and for a broad class of transmission conditions, including impedance and PML operators.

\item We identify regimes in which one-level Schwarz methods for Maxwell’s equations are weakly scalable, and we show how transmission conditions and absorption influence this behavior.

\item Numerical experiments confirm that the modal and limiting spectrum analysis accurately predicts the convergence and scalability of practical domain decomposition solvers for discretized Maxwell problems.
\end{itemize}
The main message of this paper is that, in waveguide geometries, regardless of the cross section or transmission conditions, the Schwarz method for Maxwell’s equations behaves mode-by-mode like the Schwarz method for the Helmholtz equation, which allows weak scalability results to be transferred from scalar to vector wave problems.

The structure of the paper is as follows: in Section \ref{sec:Modal} we recall elements of modal analysis for Maxwell's equations in a waveguide of general cross sections. In Section \ref{sec:DD} we present and analyse the domain decomposition method with different types of transmission conditions relying on the modal decomposition and limiting spectrum then show that the limiting spectrum is descriptive for the convergence of the algorithm even in the case of a moderate number of domains. Finally in Section \ref{sec:numeric} we illustrate our findings on a series of numerical results on the discretised problem using the edge element method.

\section{Modal decomposition in a waveguide}
\label{sec:Modal}

In a straight waveguide, solutions of Maxwell's equations can be decomposed into elementary propagating or evanescent modes. This modal structure is classical in waveguide theory and will be the main tool used later to analyze the Schwarz iteration. More precisely, the geometry allows one to separate the longitudinal variable $x$ from the transverse variables $(y,z)$, which reduces the three-dimensional Maxwell system to spectral problems posed in the cross-section of the waveguide. The corresponding modal families are the transverse electric (TE), transverse magnetic (TM), and transverse electromagnetic (TEM) modes.

In what follows, we consider the time-harmonic Maxwell equations with time convention $e^{-\imath\omega t}$:
\begin{equation}\label{Maxwell-eq}
\left\{
\begin{array}{lll}
\nabla\times \bfE = \imath\omega\mu \bfH & \text{in} & \Omega,\\
\nabla\times \bfH = -\imath\omega\varepsilon \bfE & \text{in} & \Omega,\\
\bfE\times \mathbf n = \mathbf 0 & \text{on} & \partial\Omega,
\end{array}
\right.
\qquad \Leftrightarrow \qquad
\left\{
\begin{array}{lll}
\nabla\times\nabla\times \bfE - k^2 \bfE = 0 & \text{in} & \Omega,\\
\bfE\times \mathbf n = \mathbf 0 & \text{on} & \partial\Omega,
\end{array}
\right.
\end{equation}
where $k^2=\varepsilon\mu\omega^2$, $\omega > 0$ is the pulsation, $\varepsilon=\varepsilon'-\imath\varepsilon''$ is the (possibly complex) electric permittivity, and $\mu$ is the magnetic permeability. Throughout the paper, $\varepsilon$ and $\mu$ are assumed constant. We shall mainly work with the second-order formulation, while the first-order system will be used to reconstruct the full fields from their longitudinal components. The domain is the straight infinite waveguide
\(
\Omega=\mathbb R\times S,
\)
where $S$ is the cross-section; see Figure~\ref{fig1:schema}. On the boundary we impose perfect electric conductor (PEC) conditions. We follow the classical modal analysis presented in \cite{Bonnet:2021:PGO}.

\begin{definition}
A \emph{mode} is a particular solution of Maxwell's equations of the separated form
\[
\bfE(x,y,z)=e^{\imath\beta x}\widehat{\bfE}(y,z),
\qquad
\bfH(x,y,z)=e^{\imath\beta x}\widehat{\bfH}(y,z),
\]
where $\beta\in\mathbb C$ is the longitudinal propagation constant and
$\widehat{\bfE},\widehat{\bfH}$ are vector fields defined on the cross-section $S$.
\end{definition}
\begin{figure}[h]
	\centering
	\includegraphics[height=4cm]{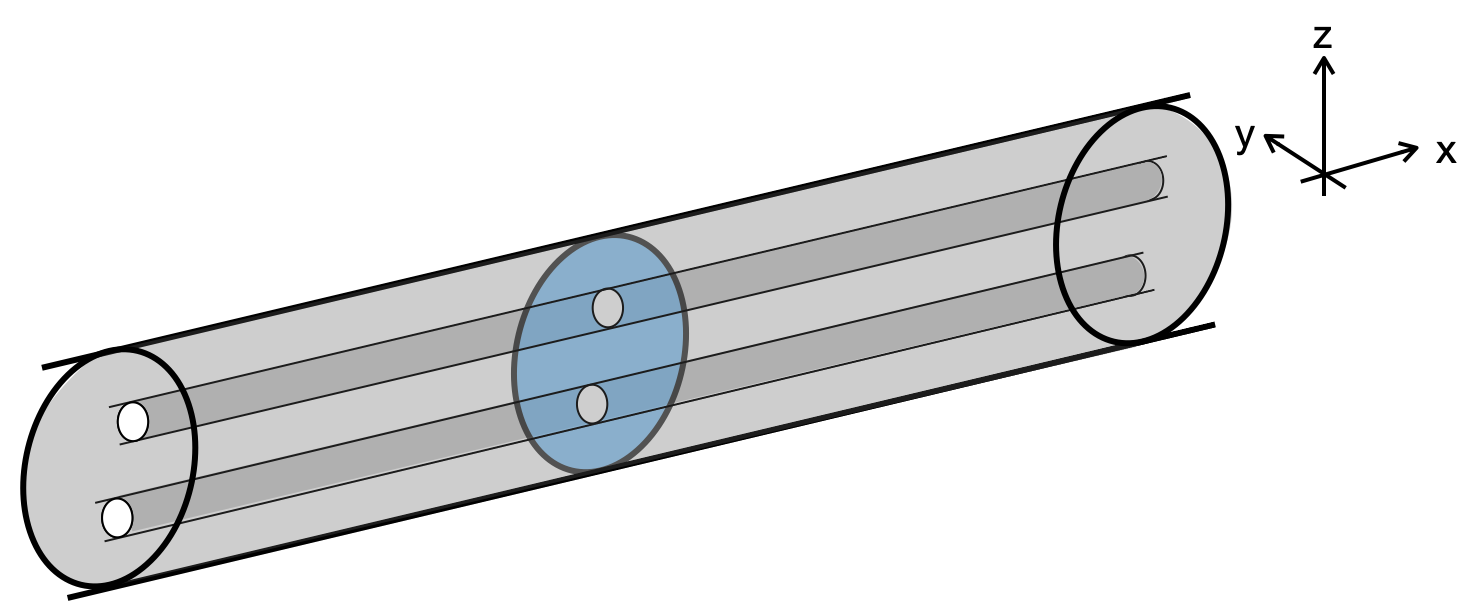}
	\caption{Scheme of the geometry $\Omega$. In blue, cross-sections of the waveguide.}\label{fig1:schema}
\end{figure}
We now characterize the admissible values of $\beta$ and the associated transverse profiles $\widehat{\bfE}$ and $\widehat{\bfH}$ for which such separated solutions satisfy Maxwell's equations. A key observation is that the electric field is divergence free. Indeed, taking the divergence of the second equation in \eqref{Maxwell-eq} gives
\(
\nabla\cdot \bfE = 0.
\)
Using the identity
\[
\nabla\times\nabla\times \bfE = \nabla(\nabla\cdot \bfE)-\Delta \bfE,
\]
it follows that each Cartesian component of $\bfE$ satisfies a scalar Helmholtz equation,
\begin{equation}
-\Delta E_i-k^2E_i=0
\qquad \text{in }\Omega,\qquad i\in\{x,y,z\}.
\end{equation}
In a waveguide, this allows the longitudinal components $E_x$ and $H_x$ to play a distinguished role: once they are known, the transverse components can be reconstructed from the first-order Maxwell system.

The PEC condition $\bfE\times \mathbf n=\mathbf 0$ on $\partial\Omega=\mathbb R\times\partial S$ implies
\begin{equation}
E_x=0 \quad \text{on } \mathbb R\times\partial S,
\qquad
E_y n_z - E_z n_y = 0 \quad \text{on } \mathbb R\times\partial S.
\end{equation}
Therefore, if $E_x=e^{\imath\beta x}\widehat E_x(y,z)$, then $\widehat E_x$ satisfies
\begin{equation}\label{pb:ex-modes}
\begin{array}{|lll}
-\Delta_{(y,z)}\widehat E_x - (k^2-\beta^2)\widehat E_x = 0 & \text{in} & S,\\
\widehat E_x = 0 & \text{on} & \partial S.
\end{array}
\end{equation}
Similarly, using the second-order formulation for $\bfH$, the longitudinal magnetic component $\widehat H_x$ satisfies
\begin{equation}\label{pb:hx-modes}
\begin{array}{|lll}
-\Delta_{(y,z)}\widehat H_x - (k^2-\beta^2)\widehat H_x = 0 & \text{in} & S,\\
\partial_{\mathbf n}\widehat H_x = 0 & \text{on} & \partial S.
\end{array}
\end{equation}

Note that knowing $\widehat{E}_x$ and $\widehat{H}_x$ one can reconstruct the full vector fields $\widehat{\bfE}$ and $\widehat{\bfH}$. Indeed, from the first order formulation of the Maxwell equations, we have:
\begin{equation}\label{eq:syst-recov}
	\begin{array}{|lll}
		\partial_y \widehat{E}_z - \partial_z \widehat{E}_y = \imath \omega \mu \widehat{H}_x,\\
		-\imath \beta \widehat{E}_z + \partial_z \widehat{E}_x = \imath \omega \mu \widehat{H}_y,\\
		\imath \beta \widehat{E}_y - \partial_y \widehat{E}_x = \imath \omega \mu \widehat{H}_z, 
	\end{array} \quad \text{and} \quad \begin{array}{|lll}
		\partial_y \widehat{H}_z - \partial_z \widehat{H}_y = -\imath \omega \varepsilon \widehat{E}_x,\\
		-\imath \beta \widehat{H}_z + \partial_z \widehat{H}_x = -\imath \omega \varepsilon  \widehat{E}_y,\\
		\imath \beta \widehat{H}_y - \partial_y \widehat{H}_x = -\imath \omega \varepsilon  \widehat{E}_z. 
	\end{array}
\end{equation}
Taking now the last two equations of both systems, we get that
\begin{equation}\label{eq:syst-recov-2}
	\left[ \begin{matrix}
		-\imath \omega \varepsilon  & 0 & 0 & \imath \beta \\
		0 & -\imath \omega \varepsilon  & -\imath \beta & 0 \\ 
		0 & \imath \beta & \imath \omega \mu & 0 \\
		-\imath \beta & 0 & 0 & \imath \omega \mu \\
	\end{matrix} \right] \left[ \begin{matrix}
		\widehat{E}_y \\ \widehat{E}_z \\ \widehat{H}_y \\ \widehat{H}_z
	\end{matrix} \right] = \left[ \begin{matrix} \partial_z \widehat{H}_x \\ - \partial_y \widehat{H}_x \\  \partial_z \widehat{E}_x \\ - \partial_y \widehat{E}_x \end{matrix} \right].
\end{equation}
We can easily see that the determinant of the matrix is not zero iff $k^2 - \beta^2 \neq 0$ and in this case the components $(\widehat{E}_{y},\widehat{E}_{z})$ and $(\widehat{H}_{y},\widehat{H}_{z})$ are uniquely defined. 

The two scalar problems \eqref{pb:ex-modes}--\eqref{pb:hx-modes} show that the modal structure is governed by Dirichlet and Neumann eigenproblems on the cross-section $S$. Depending on whether the longitudinal electric and magnetic components vanish, one obtains three mutually exclusive families of modes:
\begin{itemize}
\item Transverse Electric (TE) modes: $\widehat E_x=0$, $\widehat H_x\neq 0$;
\item Transverse Magnetic (TM) modes: $\widehat E_x\neq 0$, $\widehat H_x=0$;
\item Transverse EM (TEM) modes: $\widehat E_x=\widehat H_x=0$.
\end{itemize}
We discuss these three cases in turn.

\paragraph{TE modes: $\widehat E_x=0$ and $\widehat H_x\neq 0$.}
In this case, \eqref{pb:hx-modes} shows that $\widehat H_x$ is expanded on the Neumann eigenfunctions of the Laplacian on $S$. Hence the associated modes are of the form
\[
e^{\imath m\beta_i^{\rm TE}x}\varphi_i^{\mathcal N}(y,z),
\qquad
\beta_i^{\rm TE}=\sqrt{k^2-(\lambda_i^{\mathcal N})^2},
\qquad i\in\mathbb N^+,\quad m\in\{-,+\},
\]
where $(\lambda_i^{\mathcal N},\varphi_i^{\mathcal N})$ solves
\begin{equation}
\label{eq:ev}
\begin{array}{|lll}
-\Delta_{(y,z)}\varphi^{\mathcal N} = \lambda^{\mathcal N}\varphi^{\mathcal N} & \text{in} & S,\\[3pt]
\partial_{\mathbf n}\varphi^{\mathcal N}=0 & \text{on} & \partial S.
\end{array}
\end{equation}
The coefficients of this expansion are the \emph{modal amplitudes}, denoted by $A_i^{{\rm TE},m}$. The sign $m$ indicates the propagation direction: $m=+$ for right-going modes and $m=-$ for left-going modes. When $\beta_i^{\rm TE}\in\mathbb R$, the mode is \emph{propagative}; when $\Im(\beta_i^{\rm TE})\neq 0$, it is \emph{evanescent}.

        \begin{remark}
           The Neumann problem \eqref{eq:ev} has a countable family of eigenpairs, with $(\lambda_i^{\mathcal N})_{i\in\mathbb N}$ increasing to $+\infty$, and $(\varphi_i^{\mathcal N})_{i\in\mathbb N}$ forms an orthonormal basis of $L^2(S)$. Moreover, the mode corresponding to the constant eigenfunction $\varphi_0^{\mathcal N}=1$ does not contribute since its modal amplitude vanishes.  Indeed, using the first equation of \eqref{eq:syst-recov} of the left system, we get that:
		\[
			\imath \omega \int_{S} \widehat{H}_x dS = \int_{S}   \partial_y \widehat{E}_z - \partial_z \widehat{E}_y dS = \int_{\partial S} \widehat{E}_z n_y - \widehat{E}_y n_z  = 0, 
		\] 
		according to the BC satisfied by $\bfE$. Then, we can deduce using the orthogonality of the eigenfunctions $\varphi^\mathcal{N}_j$ that necessarily $A^{TE,m}_0 = 0$ because 
        \[
            A^{{\rm TE},m}_0 = \int_{S} \widehat{H}_x \; \varphi^\mathcal{N}_0 dS = 0.  
        \]
        \end{remark}
        
\paragraph{Transverse magnetic (TM) modes: $\widehat E_x\neq 0$ and $\widehat H_x=0$.}
In this case, \eqref{pb:ex-modes} shows that $\widehat E_x$ is expanded on the Dirichlet eigenfunctions of the Laplacian on $S$. The corresponding modes are
\[
e^{\imath m\beta_i^{\rm TM}x}\varphi_i^{\mathcal D}(y,z),
\qquad
\beta_i^{\rm TM}=\sqrt{k^2-(\lambda_i^{\mathcal D})^2},
\qquad i\in\mathbb N^+,\quad m\in\{-,+\},
\]
where
\[
\begin{array}{|lll}
-\Delta_{(y,z)}\varphi^{\mathcal D}=\lambda^{\mathcal D}\varphi^{\mathcal D} & \text{in} & S,\\
\varphi^{\mathcal D}=0 & \text{on} & \partial S.
\end{array}
\]
As for TE modes, $\beta_i^{\rm TM}$ determines whether the mode is propagative or evanescent, and the sign $m$ determines its direction of propagation.
		
\paragraph{Transverse electromagnetic (TEM) modes: $\widehat E_x=\widehat H_x=0$.}
If both longitudinal components vanish, the right-hand side of \eqref{eq:syst-recov-2} is zero. Non-trivial solutions then require
\(
(\beta^{\rm TEM})^2 = k^2.
\)
In a waveguide whose boundary $\partial S$ has $N_C$ connected components, there exist $N_C-1$ independent TEM modes; see \cite{Bonnet:2021:PGO}. In particular, if the cross-section is simply connected, no TEM mode is present.\\

Thus, any solution of Maxwell's equations in the waveguide can be expanded as a superposition of TE, TM, and, when present, TEM modes:
\begin{equation}\label{modal-dec}
	\begin{split}
	\bfE = & \sum_{i \geq 1} \sum_{m \in \{-,+\}} A^{{\rm TE},m}_i e^{\imath m \beta_i^{\rm TE} x}\; \widehat{\bfE}^{{\rm TE},m}_i + A^{{\rm TM},m}_i e^{\imath m \beta_i^{\rm TM} x} \;\widehat{\bfE}^{{\rm TM},m}_i\\
		 & + \sum_{i = 1}^{N_C-1} \sum_{m \in \{-,+\}} A^{{\rm TEM},m}_i e^{\imath m k x} \widehat{\bfE}^{{\rm TEM},m}_i.
	\end{split}	 
\end{equation}	
Here $\widehat{\bfE}_i^{{\rm TE},m}$, $\widehat{\bfE}_i^{{\rm TM},m}$, and $\widehat{\bfE}_i^{{\rm TEM},m}$ are the corresponding mode profiles on the cross-section $S$, while $A_i^{{\rm TE},\pm}$, $A_i^{{\rm TM},\pm}$, and $A_i^{{\rm TEM},\pm}$ are the associated modal amplitudes. The profiles are reconstructed from the longitudinal components through \eqref{eq:syst-recov-2}. This modal expansion will be the basis of the domain decomposition analysis in Section~\ref{sec:DD}, since it allows the Schwarz transmission operators to be analyzed mode by mode.

\begin{remark}
TEM modes are necessarily propagative, whereas TE and TM modes may be either propagative or evanescent depending on the sign of $k^2-(\lambda_i)^2$. In the remainder of the paper, we assume that the frequency is not at cut-off, namely that
\[
\beta_i^{\rm TE}\neq 0 \quad \text{and} \quad \beta_i^{\rm TM}\neq 0
\]
for all $i \in \N^+$. This excludes (for TE and TM modes) the degenerate case $k^2-\beta^2=0$, for which the reconstruction of the transverse components from \eqref{eq:syst-recov-2} is singular.
\end{remark}

The next result shows that the mode profiles can be chosen so that their tangential traces on a cross-section are independent of the propagation direction. This property will be crucial in the analysis of Schwarz transmission conditions, since the interface operators act on tangential traces.

\begin{proposition}[Tangential traces of mode profiles]\label{prop:non-depend-pm}
For $\mathbf n=(1,0,0)^t$, the mode profiles can be chosen in such a way that
\[
\widehat{\bfE}^{\rm TE,+}_i\times \mathbf n
=
\widehat{\bfE}^{\rm TE,-}_i\times \mathbf n,
\qquad
\widehat{\bfE}^{\rm TM,+}_i\times \mathbf n
=
\widehat{\bfE}^{\rm TM,-}_i\times \mathbf n,
\qquad
\forall i\in\mathbb N^+,
\]
and
\[
\widehat{\bfE}^{\rm TEM,+}_i\times \mathbf n
=
\widehat{\bfE}^{\rm TEM,-}_i\times \mathbf n,
\qquad
\forall i\in\{1,\dots,N_C-1\}.
\]
\end{proposition}

\begin{proof}
For any vector field $\bfU$, one has
\begin{equation}\label{eq:cross-prod}
\bfU\times \mathbf n=
\begin{bmatrix}
0\\ U_z\\ -U_y
\end{bmatrix}.
\end{equation}
Hence the tangential trace on a cross-section depends only on the transverse components $U_y$ and $U_z$.

For TE modes, we have $\widehat E_x=0$. Using \eqref{eq:syst-recov}, we obtain
\[
\widehat E_y=\frac{\imath\omega\mu\,\partial_z\varphi_i^{\mathcal N}}{\lambda_i^{\mathcal N}},
\qquad
\widehat E_z=-\frac{\imath\omega\mu\,\partial_y\varphi_i^{\mathcal N}}{\lambda_i^{\mathcal N}}.
\]
These expressions do not depend on the sign of the propagation constant. Therefore the TE profiles can be chosen so that
\[
\widehat{\bfE}^{\rm TE,+}_i\times \mathbf n
=
\widehat{\bfE}^{\rm TE,-}_i\times \mathbf n.
\]

The same argument applies to TEM modes, for which the longitudinal component also vanishes.

For TM modes, one has $\widehat H_x=0$, and \eqref{eq:syst-recov} yields
\[
\widehat E_y=\frac{\imath \beta_i^{\rm TM}\,\partial_y\varphi_i^{\mathcal D}}{\lambda_i^{\mathcal D}},
\qquad
\widehat E_z=\frac{\imath \beta_i^{\rm TM}\,\partial_z\varphi_i^{\mathcal D}}{\lambda_i^{\mathcal D}}.
\]
Here the transverse components do depend on the sign of $\beta_i^{\rm TM}$. However, the modal profiles are defined only up to a non-zero multiplicative constant. We can therefore choose $\widehat{\bfE}^{\rm TM,+}_i$ and $\widehat{\bfE}^{\rm TM,-}_i$ so that their transverse components coincide, and thus
\[
\widehat{\bfE}^{\rm TM,+}_i\times \mathbf n
=
\widehat{\bfE}^{\rm TM,-}_i\times \mathbf n.
\]
This concludes the proof.
\end{proof}

\begin{remark}
For TE and TEM modes, the profiles could in fact be chosen independently of the sign $\pm$. We nevertheless keep the notation $\widehat{\bfE}^{\mathrm T,\pm}_i$ for notational uniformity across the three modal families.
\end{remark}

To summarize, TE modes are associated with a Neumann eigenproblem on the cross-section, TM modes with a Dirichlet eigenproblem, and TEM modes arise only when the cross-section is multiply connected. TE and TM modes may be propagative or evanescent, whereas TEM modes are always propagative. More importantly for the sequel, the modal profiles can be normalized so that their tangential traces are independent of the propagation direction. This modal structure will allow the Schwarz interface operators to be diagonalized mode by mode; see Table~\ref{tab:notation-modal} for a summary of the notations throughout this section.

{\begin{table}[t]
\centering
\scriptsize
\setlength{\tabcolsep}{3pt}
\renewcommand{\arraystretch}{1.08}
\begin{tabular}{@{}p{0.32\linewidth}p{0.62\linewidth}@{}}
\hline
\textbf{Notation} & \textbf{Meaning} \\
\hline
$\lambda_i^{\mathcal N},\ \lambda_i^{\mathcal D}$ & Neumann/Dirichlet eigenvalues on $S$; eigenfunctions $\varphi_i^{\mathcal N},\varphi_i^{\mathcal D}$. \\
$\beta_i^{\rm TE}=\sqrt{k^2-(\lambda_i^{\mathcal N})^2}$ & TE axial wavenumber; propagative if $\mathrm{Im}\,\beta=0$, evanescent if not. \\
$\beta_i^{\rm TM}=\sqrt{k^2-(\lambda_i^{\mathcal D})^2}$ & TM axial wavenumber; propagative if $\mathrm{Im}\,\beta=0$, evanescent if not. \\
$\beta^{\rm TEM}=k$ & TEM axial wavenumber; \#modes $=N_C-1$ if $\partial S$ has $N_C$ components. \\
$\mathbf E_i^{\mathrm T,\pm}$, $A_i^{\mathrm T,\pm}$ & Mode $i$ and modal amplitude, $\mathrm T\in\{\mathrm{TE},\mathrm{TM},\mathrm{TEM}\}$. $\pm$ indicates the direction of propagation of the mode. \\
\hline
\end{tabular}
\caption{Modal notation (used in \S\ref{sec:Modal}).}
\label{tab:notation-modal}
\end{table}}

\section{Domain decomposition algorithm}
\label{sec:DD}

Section~\ref{sec:Modal} showed that Maxwell solutions in a waveguide admit a modal decomposition into TE, TM, and TEM modes. We now use this structure to analyze a one-level overlapping Schwarz method on a bounded waveguide. The key point is that the interface operators act diagonally on modal traces. As a consequence, the Schwarz iteration decouples mode by mode and reduces, for each fixed mode, to a one-dimensional nearest-neighbor recurrence between left- and right-going amplitudes. This yields a block Toeplitz iteration matrix, to which limiting-spectrum results can be applied.
We consider the truncated waveguide
\(
\widetilde\Omega=[a,b]\times S,
\)
equipped with suitable boundary conditions at the two end cross-sections $\{a\}\times S$ and $\{b\}\times S$.

{\begin{table}[t]
\centering
\scriptsize
\setlength{\tabcolsep}{3pt}
\renewcommand{\arraystretch}{1.08}
\begin{tabular}{@{}p{0.32\linewidth}p{0.62\linewidth}@{}}
\hline
\textbf{Notation} & \textbf{Meaning} \\
\hline
$\widetilde{\Omega}=[a,b]\times S$ & Truncated domain; $\Omega_l=[a_l,b_l]\times S$, $l=1,\dots,N$. \\
$L,\ \delta$ & Core length and overlap: $b_{l-1}-a_l=\delta$, $b_l-a_l=L+2\delta$. \\
$\Gamma_{l,l\pm1}$ & Interfaces $\{x=a_l\}\times S$ and $\{x=b_l\}\times S$; $\mathbf n=(\pm1,0,0)^{\top}$. \\
$\mathcal L,\ \mathcal B,\ \mathcal T$ & Operators in \eqref{pb:Schwarz-ite}--\eqref{eq:IC}; $\mathcal T$ diagonal on modes. \\
$\lambda_i^{\mathrm T}$ & Modal symbol of $\mathcal T$: $\mathcal T(\mathbf E_i^{\mathrm T,\pm}\!\times\mathbf n)=\lambda_i^{\mathrm T}(\mathbf E_i^{\mathrm T,\pm}\!\times\mathbf n)$. \\
\hline
\end{tabular}
\caption{Notation introduced in \S\ref{sec:DD}.}
\label{tab:notation-dd}
\end{table}}

We decompose $\widetilde\Omega$ into $N$ overlapping subdomains
\(
\Omega_l=[a_l,b_l]\times S,
\)
with
\[
a_1=a,\qquad b_N=b,\qquad b_{l-1}-a_l=\delta>0,\qquad b_l-a_l=L+2\delta.
\]
Thus each subdomain has the same cross-section $S$, core length $L$, and overlap $\delta$ with its neighbors; see Figure~\ref{fig1:schemaDD} and Table~\ref{tab:notation-dd}.
	\begin{figure}[h]
		\centering
		\includegraphics[height=4cm]{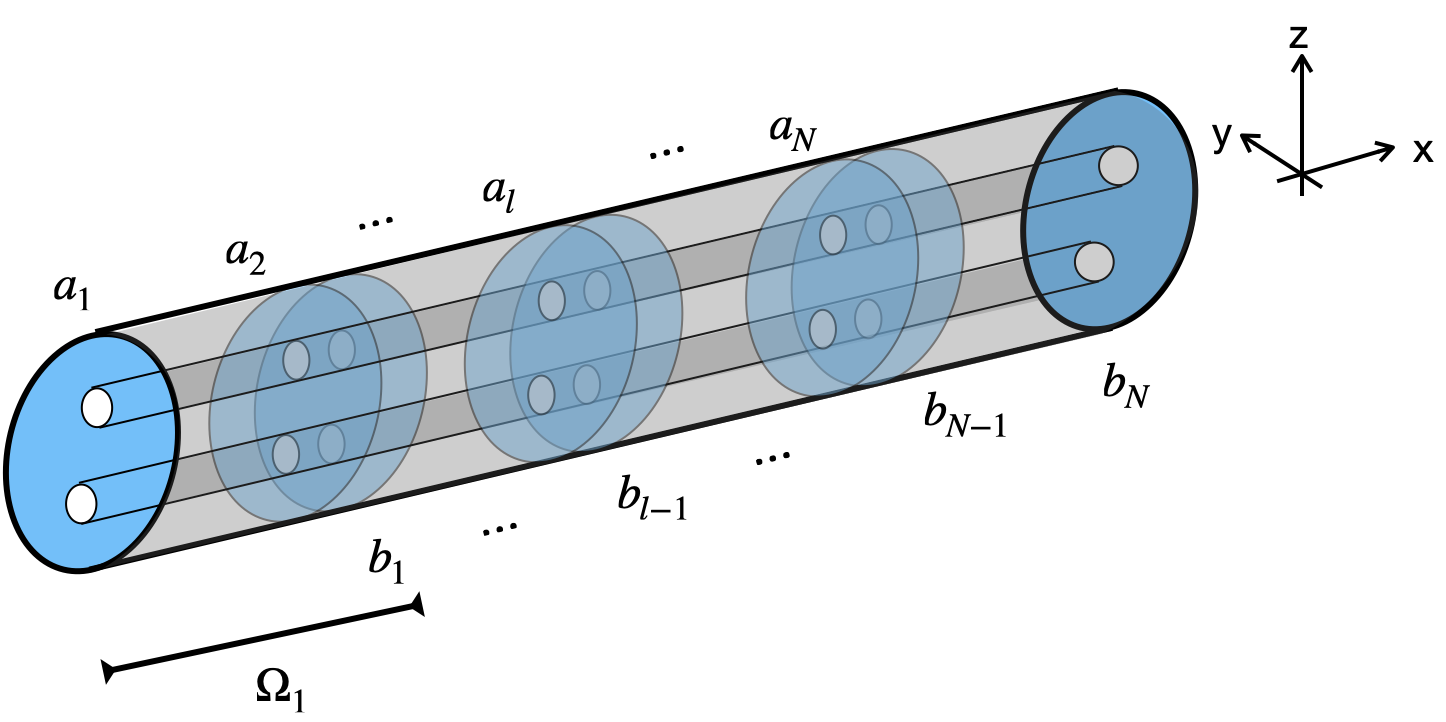}
		\caption{A bounded waveguide $\widetilde{\Omega}$ and its decomposition into subdomains.}\label{fig1:schemaDD}
	\end{figure}
\begin{definition}[Overlapping Schwarz iteration]
Let
\[
\mathcal L(\mathbf E):=\nabla\times\nabla\times \mathbf E-k^2\mathbf E
\]
be the Maxwell operator, and let
\begin{equation}
\label{eq:IC}
\mathcal B(\mathbf E):=((\nabla\times\mathbf E)\times \mathbf n)\times \mathbf n
+\mathcal T(\mathbf E\times\mathbf n)
\end{equation}
be the interface operator, where $\mathcal T$ is a transmission operator to be specified later. Starting from an initial guess, the Schwarz iteration computes, for each subdomain $\Omega_l$ and each iteration $n\ge 1$, a field $\mathbf E^{n,l}$ satisfying
\begin{equation}
\label{pb:Schwarz-ite}
\left\{
\begin{array}{rll}
\mathcal L(\mathbf E^{n,l})=\mathbf f_l & \text{in} & \Omega_l,\\
\mathcal B(\mathbf E^{n,l})=\mathcal B(\mathbf E^{n-1,l-1}) & \text{on} & \Gamma_{l,l-1},\\
\mathcal B(\mathbf E^{n,l})=\mathcal B(\mathbf E^{n-1,l+1}) & \text{on} & \Gamma_{l,l+1},\\
\mathbf E^{n,l}\times\mathbf n=\mathbf 0 & \text{on} & \partial\Omega_l\setminus(\Gamma_{l,l-1}\cup\Gamma_{l,l+1}).
\end{array}
\right.
\end{equation}
At the two ends of the guide we prescribe boundary data $\mathbf g_1$ and $\mathbf g_N$ through
\[
\mathcal B(\mathbf E^{n-1,1})=\mathbf g_1,
\qquad
\mathcal B(\mathbf E^{n-1,N})=\mathbf g_N.
\]
\end{definition}
    
\begin{remark}
For instance, choosing $\mathcal T=-\imath k\,\mathrm{Id}$ yields the classical first-order absorbing (impedance) condition. Such conditions are commonly used both to truncate unbounded waveguides and as transmission conditions in domain decomposition; see, e.g., \cite{haddar:2008:GIBC,hall:2004:SABCMaxwell}. In scattering settings, the end conditions are typically inhomogeneous, so that $\mathbf g_1$ and $\mathbf g_N$ need not vanish.
\end{remark}

In the convergence analysis below, we consider the homogeneous setting
\(
\mathbf f_l=\mathbf 0,\, \mathbf g_1=\mathbf g_N=\mathbf 0,
\)
since the Schwarz iteration for the error satisfies exactly this problem. Our goal is therefore to characterize how the error propagates across interfaces and decays with the iteration count.

The convergence factor of algorithm \eqref{pb:Schwarz-ite} will be computed using a Fourier analysis and different types of transmissions conditions will be considered. 
The central observation is that the interface operator $\mathcal B$ acts diagonally on modal traces. We first prove this for the curl part of $\mathcal B$, and then verify it for the transmission operators used in practice. (see also the analysis in \cite{Dolean:2015:ETC}).

For each modal family ${\rm T}\in\{{\rm TE,TM,TEM}\}$, let
\[
\mathbf E_i^{{\rm T},\pm}(x,y,z):=
e^{\pm \imath \beta_i^{\rm T}x}\,\widehat{\mathbf E}_i^{{\rm T},\pm}(y,z),
\]
with the convention $\beta_i^{\rm TEM}=k$.

\begin{proposition}[Diagonalization of the tangential curl operator]\label{prop:BC-TE-TM}
Let $\mathbf n=(1,0,0)^t$ be the normal to a cross-section of the waveguide. Then the operator
\[
\mathbf E \;\longmapsto\; ((\nabla\times \mathbf E)\times \mathbf n)\times \mathbf n
\]
acts diagonally on TE, TM, and TEM modes. More precisely, for every modal index $i$,
\[
((\nabla\times \mathbf E_i^{{\rm T},\pm})\times \mathbf n)\times \mathbf n
=
\mu_i^{\rm T,\pm}\,(\mathbf E_i^{{\rm T},\pm}\times \mathbf n),
\]
where
\[
\mu_i^{\rm TE,\pm}=\pm \imath \beta_i^{\rm TE},
\qquad
\mu_i^{\rm TM,\pm}=\pm \imath \frac{k^2}{\beta_i^{\rm TM}},
\qquad
\mu_i^{\rm TEM,\pm}=\pm \imath k.
\]
\end{proposition}

\begin{proof}
For any field $\mathbf E=(E_x,E_y,E_z)^t$, one has
\begin{equation}\label{eq:curl-cross-n}
(\nabla\times \mathbf E)\times \mathbf n
=
\begin{bmatrix}
0\\
\partial_x E_y-\partial_y E_x\\
\partial_x E_z-\partial_z E_x
\end{bmatrix}.
\end{equation}

\smallskip
\noindent\textbf{TE and TEM modes.}
For TE modes, $E_x^{\rm TE,\pm}=0$. Since
\[
\mathbf E^{\rm TE,\pm}_i
=
e^{\pm \imath \beta_i^{\rm TE}x}\widehat{\mathbf E}_i^{\rm TE,\pm},
\]
formula \eqref{eq:curl-cross-n} gives
\[
(\nabla\times \mathbf E_i^{\rm TE,\pm})\times \mathbf n
=
\pm \imath \beta_i^{\rm TE}\,\mathbf E_i^{\rm TE,\pm}.
\]
Taking once more the cross product with $\mathbf n$ yields
\[
((\nabla\times \mathbf E_i^{\rm TE,\pm})\times \mathbf n)\times \mathbf n
=
\pm \imath \beta_i^{\rm TE}\,(\mathbf E_i^{\rm TE,\pm}\times \mathbf n).
\]
The same argument applies to TEM modes, since their longitudinal electric component also vanishes and $\beta_i^{\rm TEM}=k$.

\smallskip
\noindent\textbf{TM modes.}
For TM modes, one has $H_x^{\rm TM,\pm}=0$ and
\[
\mathbf H_i^{\rm TM,\pm}
=
e^{\pm \imath \beta_i^{\rm TM}x}\widehat{\mathbf H}_i^{\rm TM,\pm}.
\]
Applying \eqref{eq:curl-cross-n} to $\mathbf H_i^{\rm TM,\pm}$ gives
\[
(\nabla\times \mathbf H_i^{\rm TM,\pm})\times \mathbf n
=
\pm \imath \beta_i^{\rm TM}\,\mathbf H_i^{\rm TM,\pm}.
\]
Using the first-order Maxwell equation
\[
\nabla\times \mathbf H_i^{\rm TM,\pm}
=
-\imath \omega\varepsilon\,\mathbf E_i^{\rm TM,\pm},
\]
we obtain
\[
-\imath \omega\varepsilon\,(\mathbf E_i^{\rm TM,\pm}\times \mathbf n)
=
\pm \imath \beta_i^{\rm TM}\,\mathbf H_i^{\rm TM,\pm}.
\]
Now, from
\[
\nabla\times \mathbf E_i^{\rm TM,\pm}
=
\imath \omega\mu\,\mathbf H_i^{\rm TM,\pm},
\]
it follows that
\[
\nabla\times \mathbf E_i^{\rm TM,\pm}
=
\mp \imath \frac{k^2}{\beta_i^{\rm TM}}\,
(\mathbf E_i^{\rm TM,\pm}\times \mathbf n).
\]
Taking the tangential trace gives
\[
((\nabla\times \mathbf E_i^{\rm TM,\pm})\times \mathbf n)\times \mathbf n
=
\mp \imath \frac{k^2}{\beta_i^{\rm TM}}
\bigl((\mathbf E_i^{\rm TM,\pm}\times \mathbf n)\times \mathbf n\bigr)\times \mathbf n.
\]
Finally, since
\[
\bigl((\mathbf u\times \mathbf n)\times \mathbf n\bigr)\times \mathbf n
=
-\mathbf u\times \mathbf n
\]
for any vector field $\mathbf u$, we conclude that
\[
((\nabla\times \mathbf E_i^{\rm TM,\pm})\times \mathbf n)\times \mathbf n
=
\pm \imath \frac{k^2}{\beta_i^{\rm TM}}\,
(\mathbf E_i^{\rm TM,\pm}\times \mathbf n).
\]
This proves the result.
\end{proof}

\subsection{Transmission conditions}
\label{subsec:TC}

Proposition~\ref{prop:BC-TE-TM} shows that the curl part of the interface operator is diagonal on modal traces. We now verify that the same is true for the transmission operator $\mathcal T$, namely that
\begin{equation}\label{eq:op-T-diag}
\mathcal T(\mathbf E_i^{{\rm T},\pm}\times \mathbf n)
=
\lambda_i^{{\rm T},\pm}\,(\mathbf E_i^{{\rm T},\pm}\times \mathbf n),
\qquad
{\rm T}\in\{{\rm TE,TM,TEM}\}.
\end{equation}
This property is satisfied by the two classes of transmission conditions considered in this paper: impedance conditions and PML-based conditions. We derive below the corresponding modal symbols $\lambda_i^{\rm T}$.

\paragraph{Impedance conditions (first-order absorbing boundary conditions).}
Consider
\begin{equation}\label{def:op-T-ik}
\mathcal T : \mathbf E\times \mathbf n \longmapsto -\imath k\,(\mathbf E\times \mathbf n).
\end{equation}
This is simply the identity operator multiplied by $-\imath k$, and therefore \eqref{eq:op-T-diag} holds immediately with
\[
\lambda_i^{{\rm T},\pm}=-\imath k
\]
for all TE and TM modes ($i\in\mathbb N^+$), and for all TEM modes ($i=1,\dots,N_C-1$). This choice corresponds to the standard first-order absorbing, or impedance, condition entering the interface operator $\mathcal B$.

\paragraph{PML transmission conditions.}
Perfectly matched layers (PMLs) are widely used to absorb outgoing waves in wave propagation problems \cite{bramble2007analysis}. In the present context, they can also be used as transmission conditions in domain decomposition methods \cite{Bootland:2022:NAP,royer2022non}. Their effect is to transform propagative waves into exponentially decaying ones without reflection at the continuous level.

To define the corresponding operator $\mathcal T$, we extend each subdomain $\Omega_j$ by a PML of width $\ell$, truncated by the homogeneous boundary condition
\(\mathbf E\times \mathbf n = 0 \) {at} \(x=b_j+\ell\)  and \(x=a_j-\ell.\)
We use the complex stretching
\[
\widetilde x(x)=
\begin{cases}
x, & x\in[a_j,b_j],\\
x+\imath\sigma(x-b_j), & x\in[b_j,b_j+\ell],\\
x-\imath\sigma(x-a_j), & x\in[a_j-\ell,a_j],
\end{cases}
\]
where $\sigma>0$ is taken constant for simplicity.\footnote{The same analysis extends to a positive increasing PML profile $\sigma(x)$.}

Consider, for instance, the right PML $x\ge b_j$. For each modal family ${\rm T}\in\{{\rm TE,TM,TEM}\}$, the field in the layer is a linear combination of outgoing and reflected modal contributions,
\[
\mathbf E
=
\sum_{T\in\{TE,TM,TEM\}}\sum_{i} A_i^{{\rm T},+}
\left(
e^{\imath \beta_i^{\rm T}(1+\imath\sigma)(x-b_j)}
-
e^{2\imath \beta_i^{\rm T}(1+\imath\sigma)\ell}
e^{-\imath \beta_i^{\rm T}(1+\imath\sigma)(x-b_j)}
\right)
\widehat{\mathbf E}_i^{{\rm T},+},
\]
where $i\in\mathbb N^+$ for ${\rm T}\in\{{\rm TE,TM}\}$ and $i=1,\dots,N_C-1$ for ${\rm T}={\rm TEM}$, with the convention $\beta_i^{\rm TEM}=k$. Evaluating this expression at $x=b_j$ gives
\[
\mathbf E\times \mathbf n
=
\sum_{T\in\{TE,TM,TEM\}}\sum_i
A_i^{{\rm T},+}
\left(
1-
e^{2\imath \beta_i^{\rm T}(1+\imath\sigma)\ell}
\right)
\widehat{\mathbf E}_i^{{\rm T},+}\times \mathbf n.
\]
Using Proposition~\ref{prop:BC-TE-TM}, we also obtain
\[
((\nabla\times \mathbf E)\times \mathbf n)\times \mathbf n
=
\sum_{T\in\{TE,TM,TEM\}}\sum_i
\mu_i^{{\rm T},+}\,
A_i^{{\rm T},+}
\left(
1+
e^{2\imath \beta_i^{\rm T}(1+\imath\sigma)\ell}
\right)
\widehat{\mathbf E}_i^{{\rm T},+}\times \mathbf n.
\]
Hence the PML transmission operator, defined by mapping 
\[\mathbf E\times \mathbf n \longmapsto ((\nabla\times \mathbf E)\times \mathbf n)\times \mathbf n\] 
at the interface, is diagonal on modal traces and its modal symbol is
\begin{equation}\label{eq:lambdaPML-compact}
\lambda_i^{\rm T,+}
=
-\mu_i^{{\rm T},+}\,
\frac{1+e^{2\imath \beta_i^{\rm T}(1+\imath\sigma)\ell}}
     {1-e^{2\imath \beta_i^{\rm T}(1+\imath\sigma)\ell}}.
\end{equation}
Using the values of $\mu_i^{{\rm T},+}$ from Proposition~\ref{prop:BC-TE-TM}, this yields
\begin{align}
\lambda_i^{\rm TE,+}
&=
-\imath \beta_i^{\rm TE}
\frac{1+e^{2\imath \beta_i^{\rm TE}(1+\imath\sigma)\ell}}
     {1-e^{2\imath \beta_i^{\rm TE}(1+\imath\sigma)\ell}},
\\
\lambda_i^{\rm TM,+}
&=
-\imath \frac{k^2}{\beta_i^{\rm TM}}
\frac{1+e^{2\imath \beta_i^{\rm TM}(1+\imath\sigma)\ell}}
     {1-e^{2\imath \beta_i^{\rm TM}(1+\imath\sigma)\ell}},
\\
\lambda_i^{\rm TEM,+}
&=
-\imath k
\frac{1+e^{2\imath k(1+\imath\sigma)\ell}}
     {1-e^{2\imath k(1+\imath\sigma)\ell}}.
\end{align}

By Proposition~\ref{prop:non-depend-pm}, the tangential traces of the mode profiles do not depend on the propagation direction. Therefore the modal symbol is the same for the $+$ and $-$ modes, and in what follows we simply write $\lambda_i^{\rm T}$.

\begin{remark}[Exactness of PML conditions]
If $\sigma>0$ and $\ell\to+\infty$, then
\[
e^{2\imath \beta_i^{\rm T}(1+\imath\sigma)\ell}\to 0,
\]
so that the PML symbol tends to
\(
\lambda_i^{\rm T}=-\mu_i^{{\rm T},+}.
\)
In other words, $\mathcal T$ converges to \emph{the exact Dirichlet-to-Neumann operator}, which provides the transparent boundary condition for the semi-infinite waveguide.
\end{remark}

To summarize, both impedance and PML transmission conditions satisfy the modal diagonalization property \eqref{eq:op-T-diag}. Their modal symbols are explicit and will be the main inputs of the spectral analysis of the Schwarz iteration in the next subsection.

\subsection{Analysis of the Schwarz algorithm}

Because both the tangential curl operator and the transmission operator are diagonal on modal traces, the Schwarz iteration decouples mode by mode. For each fixed mode \(i\), the iteration therefore reduces to a relation between the left- and right-going modal amplitudes on neighboring subdomains. We first derive this per-mode recurrence, then rewrite it as a nearest-neighbor iteration with constant \(2\times 2\) blocks, which reveals the block Toeplitz structure of the global iteration matrix. This will allow us to apply limiting-spectrum results and discuss weak scalability.

\subsubsection{Per-mode interface relations}

For each modal family ${\rm T}\in\{{\rm TE,TM,TEM}\}$, let
\[
\eta_i^{\rm T}:=\mu_i^{{\rm T},+},
\]
that is,
\[
\eta_i^{\rm TE}=\imath\beta_i^{\rm TE},\qquad
\eta_i^{\rm TM}=\imath\frac{k^2}{\beta_i^{\rm TM}},\qquad
\eta_i^{\rm TEM}=\imath k.
\]

\begin{proposition}[Per-mode evolution]\label{prop:mode-ite}
Fix a modal family ${\rm T}\in\{{\rm TE,TM,TEM}\}$ and a mode index $i$. Assume that, in each subdomain $\Omega_l$, the Schwarz iterate at step $n-1$ is a superposition of the two traveling modes
\[
\mathbf E^{n-1,l}
=
e^{\imath \beta_i^{\rm T}x}A_i^{{\rm T},+,n-1,l}\widehat{\mathbf E}_i^{{\rm T},+}
+
e^{-\imath \beta_i^{\rm T}x}A_i^{{\rm T},-,n-1,l}\widehat{\mathbf E}_i^{{\rm T},-}.
\]
Then the next Schwarz iterate in $\Omega_l$ is of the same form,
\[
\mathbf E^{n,l}
=
e^{\imath \beta_i^{\rm T}x}A_i^{{\rm T},+,n,l}\widehat{\mathbf E}_i^{{\rm T},+}
+
e^{-\imath \beta_i^{\rm T}x}A_i^{{\rm T},-,n,l}\widehat{\mathbf E}_i^{{\rm T},-},
\]
and the modal amplitudes satisfy
\begin{equation}\label{eq:iterT-generic}
M_{\rm T}^{l,l}(i)\,\mathbf A^{{\rm T},n,l}
=
M_{\rm T}^{l,l-1}(i)\,\mathbf A^{{\rm T},n-1,l-1}
+
M_{\rm T}^{l,l+1}(i)\,\mathbf A^{{\rm T},n-1,l+1},
\end{equation}
where
\[
\mathbf A^{{\rm T},n,l}
=
\begin{bmatrix}
A_i^{{\rm T},-,n,l}\\[2pt]
A_i^{{\rm T},+,n,l}
\end{bmatrix},
\]
and
\[
M_{\rm T}^{l,l-1}(i):=
\begin{bmatrix}
(\eta_i^{\rm T}+\lambda_i^{\rm T})e^{-\imath\beta_i^{\rm T}a_l}
&
(-\eta_i^{\rm T}+\lambda_i^{\rm T})e^{\imath\beta_i^{\rm T}a_l}
\\
0 & 0
\end{bmatrix},
\]
\[
M_{\rm T}^{l,l+1}(i):=
\begin{bmatrix}
0 & 0
\\
(-\eta_i^{\rm T}+\lambda_i^{\rm T})e^{-\imath\beta_i^{\rm T}b_l}
&
(\eta_i^{\rm T}+\lambda_i^{\rm T})e^{\imath\beta_i^{\rm T}b_l}
\end{bmatrix},
\]
and
\[
M_{\rm T}^{l,l}(i)=M_{\rm T}^{l,l-1}(i)+M_{\rm T}^{l,l+1}(i).
\]
\end{proposition}
\begin{proof}
For readability, we omit the modal index $i$ and the superscript ${\rm T}$. In $\Omega_l$, we seek a solution of the form
\[
\mathbf E^{n,l}
=
e^{\imath \beta x}A^{+,n,l}\widehat{\mathbf E}^{+}
+
e^{-\imath \beta x}A^{-,n,l}\widehat{\mathbf E}^{-}.
\]
Since each modal field solves the homogeneous Maxwell equation, this ansatz satisfies
\[
\mathcal L(\mathbf E^{n,l})=\mathbf 0
\qquad\text{in }\Omega_l.
\]
It therefore remains to enforce the interface conditions. On the left interface $\Gamma_{l,l-1}=\{x=a_l\}\times S$, the outward normal is $\mathbf n=(-1,0,0)^t$. Using the definition of $\mathcal B$, the diagonalization of the tangential curl operator from Proposition~\ref{prop:BC-TE-TM}, and the diagonalization of $\mathcal T$, we obtain
\[
(-\eta+\lambda)e^{\imath\beta a_l}A^{+,n,l}(\widehat{\mathbf E}^{+}\times\mathbf n)
+
(\eta+\lambda)e^{-\imath\beta a_l}A^{-,n,l}(\widehat{\mathbf E}^{-}\times\mathbf n)
\]
\[
=
(-\eta+\lambda)e^{\imath\beta a_l}A^{+,n-1,l-1}(\widehat{\mathbf E}^{+}\times\mathbf n)
+
(\eta+\lambda)e^{-\imath\beta a_l}A^{-,n-1,l-1}(\widehat{\mathbf E}^{-}\times\mathbf n).
\]
By Proposition~\ref{prop:non-depend-pm}, the tangential traces of the $\pm$ profiles coincide, so this simplifies to
\[
(-\eta+\lambda)e^{\imath\beta a_l}A^{+,n,l}
+
(\eta+\lambda)e^{-\imath\beta a_l}A^{-,n,l}
\]
\[
=
(-\eta+\lambda)e^{\imath\beta a_l}A^{+,n-1,l-1}
+
(\eta+\lambda)e^{-\imath\beta a_l}A^{-,n-1,l-1}.
\]
Similarly, on the right interface $\Gamma_{l,l+1}=\{x=b_l\}\times S$, where $\mathbf n=(1,0,0)^t$, we obtain
\[
(\eta+\lambda)e^{\imath\beta b_l}A^{+,n,l}
+
(-\eta+\lambda)e^{-\imath\beta b_l}A^{-,n,l}
\]
\[
=
(\eta+\lambda)e^{\imath\beta b_l}A^{+,n-1,l+1}
+
(-\eta+\lambda)e^{-\imath\beta b_l}A^{-,n-1,l+1}.
\]
These two scalar relations are exactly equivalent to \eqref{eq:iterT-generic} which concludes the proof.
\end{proof}

\begin{remark}
The end conditions at the extremal subdomains are incorporated by setting
\[
M_{\rm T}^{1,0}(i)=0,
\qquad
M_{\rm T}^{N,N+1}(i)=0.
\]
\end{remark}

\subsubsection{From modal amplitudes to a block Toeplitz iteration}

Proposition~\ref{prop:mode-ite} expresses the Schwarz iteration, for each fixed mode, in terms of the modal amplitudes in neighboring subdomains. We now introduce interface variables that eliminate the local amplitudes and reveal the translation-invariant structure of the iteration.

\begin{proposition}[Nearest-neighbor form of the modal iteration]\label{prop:R-ite}
For each modal family ${\rm T}\in\{{\rm TE,TM,TEM}\}$ and each mode index $i$, define
\[
\mathbf R^{{\rm T},n,l}(i):=
M_{\rm T}^{l,l}(i)\,\mathbf A^{{\rm T},n,l}(i)
=
\begin{bmatrix}
R^{{\rm T},-,n,l}(i)\\[2pt]
R^{{\rm T},+,n,l}(i)
\end{bmatrix}.
\]
Then the recurrence \eqref{eq:iterT-generic} can be rewritten as
\begin{equation}\label{eq:R-iteration}
\mathbf R^{{\rm T},n,l}(i)
=
K_{\rm T}^-(i)\,\mathbf R^{{\rm T},n-1,l-1}(i)
+
K_{\rm T}^+(i)\,\mathbf R^{{\rm T},n-1,l+1}(i),
\end{equation}
where the \(2\times 2\) matrices \(K_{\rm T}^\pm(i)\) are independent of the subdomain index \(l\) and are given by
\[
K_{\rm T}^+(i)=
\begin{bmatrix}
0 & 0\\
a_i^{\rm T} & b_i^{\rm T}
\end{bmatrix},
\qquad
K_{\rm T}^-(i)=
\begin{bmatrix}
b_i^{\rm T} & a_i^{\rm T}\\
0 & 0
\end{bmatrix},
\]
with
\begin{equation}\label{eq:aT}
a_i^{\rm T}
=
\frac{\bigl((\lambda_i^{\rm T})^2-(\eta_i^{\rm T})^2\bigr)
\left(e^{\imath\beta_i^{\rm T}(L+\delta)}-e^{-\imath\beta_i^{\rm T}(L+\delta)}\right)}
{\left(\eta_i^{\rm T}+\lambda_i^{\rm T}\right)^2e^{\imath\beta_i^{\rm T}(L+2\delta)}
-
\left(-\eta_i^{\rm T}+\lambda_i^{\rm T}\right)^2e^{-\imath\beta_i^{\rm T}(L+2\delta)}},
\end{equation}
\begin{equation}\label{eq:bT}
b_i^{\rm T}
=
\frac{\left(\eta_i^{\rm T}+\lambda_i^{\rm T}\right)^2e^{\imath\beta_i^{\rm T}\delta}
-
\left(\lambda_i^{\rm T}-\eta_i^{\rm T}\right)^2e^{-\imath\beta_i^{\rm T}\delta}}
{\left(\eta_i^{\rm T}+\lambda_i^{\rm T}\right)^2e^{\imath\beta_i^{\rm T}(L+2\delta)}
-
\left(-\eta_i^{\rm T}+\lambda_i^{\rm T}\right)^2e^{-\imath\beta_i^{\rm T}(L+2\delta)}}.
\end{equation}
\end{proposition}

\begin{proof}
For readability, we omit the modal index \(i\) and the superscript \({\rm T}\). Starting from \eqref{eq:iterT-generic}, we have
\[
M^{l,l}\mathbf A^{n,l}
=
M^{l,l-1}\mathbf A^{n-1,l-1}
+
M^{l,l+1}\mathbf A^{n-1,l+1}.
\]
By definition of \(\mathbf R^{n,l}\), this becomes
\[
\mathbf R^{n,l}
=
M^{l,l-1}\bigl(M^{l-1,l-1}\bigr)^{-1}\mathbf R^{n-1,l-1}
+
M^{l,l+1}\bigl(M^{l+1,l+1}\bigr)^{-1}\mathbf R^{n-1,l+1}.
\]
We therefore set
\[
K^-:=M^{l,l-1}\bigl(M^{l-1,l-1}\bigr)^{-1},
\qquad
K^+:=M^{l,l+1}\bigl(M^{l+1,l+1}\bigr)^{-1}.
\]

It remains to show that these matrices do not depend on \(l\). Since each subdomain has the same length \(L+2\delta\), one has
\[
\det\bigl(M^{l\pm1,l\pm1}\bigr)
=
\left(\eta+\lambda\right)^2e^{\imath\beta(L+2\delta)}
-
\left(-\eta+\lambda\right)^2e^{-\imath\beta(L+2\delta)}
=:D,
\]
which is independent of \(l\). Using the relations
\[
b_{l+1}-b_l=L+\delta,
\qquad
b_l-a_{l+1}=\delta,
\]
a direct computation gives
\[
K^+
=
\frac{1}{D}
\begin{bmatrix}
0 & 0\\
(\lambda^2-\eta^2)\bigl(e^{\imath\beta(L+\delta)}-e^{-\imath\beta(L+\delta)}\bigr)
&
(\eta+\lambda)^2e^{\imath\beta\delta}-(\lambda-\eta)^2e^{-\imath\beta\delta}
\end{bmatrix}.
\]
Similarly, using
\[
b_{l-1}-a_l=\delta,
\qquad
a_l-a_{l-1}=L+\delta,
\]
we obtain
\[
K^-
=
\frac{1}{D}
\begin{bmatrix}
(\eta+\lambda)^2e^{\imath\beta\delta}-(\lambda-\eta)^2e^{-\imath\beta\delta}
&
(\lambda^2-\eta^2)\bigl(e^{\imath\beta(L+\delta)}-e^{-\imath\beta(L+\delta)}\bigr)
\\
0 & 0
\end{bmatrix}.
\]
This yields exactly the coefficients \(a_i^{\rm T}\) and \(b_i^{\rm T}\) in \eqref{eq:aT}--\eqref{eq:bT}.
\end{proof}

The preceding proposition shows that, for each fixed mode, the Schwarz iteration reduces to a nearest-neighbor recurrence with constant \(2\times2\) blocks. This makes explicit the close connection with the scalar Helmholtz case.

\begin{corollary}[Maxwell--Helmholtz dictionary]\label{cor:MH-dictionary}
Let \(a_i^{\rm T}\) and \(b_i^{\rm T}\) be the coefficients defined in Proposition~\ref{prop:R-ite}. Then:
\begin{enumerate}
\item[(a)] \textbf{Impedance transmission.} If \(\mathcal T=-\imath k\,\mathrm{Id}\), then for modes with coincident axial wavenumbers \(\beta_i^{\rm TE}=\beta_i^{\rm TM}\),
\[
a_i^{\rm TE}
=
-\,a_i^{\rm TM},
\qquad
b_i^{\rm TE}
=
b_i^{\rm TM},
\]
after the substitution
\[
\eta_i^{\rm TE}
\longleftrightarrow
\eta_i^{\rm TM}
=
\imath \frac{k^2}{\beta_i^{\rm TM}}.
\]
Hence TE and TM modes have the same limiting convergence factor, since the sign change in \(a_i^{\rm T}\) only exchanges the two quantities
\(
|a_i^{\rm T}+b_i^{\rm T}|
\)
and
\(
|a_i^{\rm T}-b_i^{\rm T}|.
\)

\item[(b)] \textbf{PML transmission.} With the PML symbols from \eqref{eq:lambdaPML-compact}, one has the same correspondence between TE and TM coefficients, again yielding the same limiting convergence factor for the two families.
\end{enumerate}
Moreover, TEM modes behave as TE modes with \(\beta_i^{\rm TEM}=k\). Therefore, mode by mode, the Maxwell iteration reduces to the same block structure as in the Helmholtz case.
\end{corollary}

For each modal family \({\rm T}\in\{{\rm TE,TM,TEM}\}\) and each mode index \(i\), define the global interface vector
\[
\mathbf R^{{\rm T},n}(i)
:=
\begin{bmatrix}
R_i^{{\rm T},-,n,1}\\
R_i^{{\rm T},+,n,1}\\
\vdots\\
R_i^{{\rm T},-,n,N}\\
R_i^{{\rm T},+,n,N}
\end{bmatrix}
\in \mathbb C^{2N}.
\]
Then \eqref{eq:R-iteration} is equivalent to the recurrence
\[
\mathbf R^{{\rm T},n+1}(i)
=
I^{\rm T}(i)\,\mathbf R^{{\rm T},n}(i),
\]
where \(I^{\rm T}(i)\in\mathbb C^{2N\times 2N}\) is the block tridiagonal matrix
\[
I^{\rm T}(i)
=
\begin{bmatrix}
0_{2\times2} & K_{\rm T}^+(i) &        &        \\
K_{\rm T}^-(i) & 0_{2\times2} & \ddots &        \\
        & \ddots & \ddots & K_{\rm T}^+(i) \\
        &        & K_{\rm T}^-(i) & 0_{2\times2}
\end{bmatrix}.
\]
Since the blocks \(K_{\rm T}^\pm(i)\) do not depend on the subdomain index \(l\), the matrix \(I^{\rm T}(i)\) has a block Toeplitz structure. This is the key property that allows one to study convergence and weak scalability through the spectrum of \(I^{\rm T}(i)\) as the number of subdomains increases.

\subsubsection{Limiting spectrum and weak scalability}

For each fixed mode \(i\) and each modal family \({\rm T}\in\{{\rm TE,TM,TEM}\}\), Proposition~\ref{prop:R-ite} shows that the Schwarz iteration is governed by a block Toeplitz matrix \(I^{\rm T}(i)\). We may therefore apply the limiting-spectrum analysis of \cite{Bootland:2022:APS}.

\begin{theorem}[Limiting convergence factor]\label{thm:spectre-limit}
Assume that \(a_i^{\rm T}\neq 0\) and \(b_i^{\rm T}\neq 0\). Then the spectral radius of the modal iteration matrix \(I^{\rm T}(i)\) satisfies
\[
\lim_{N\to+\infty}\rho\bigl(I^{\rm T}(i)\bigr)
=
\max\Bigl(
|a_i^{\rm T}+b_i^{\rm T}|,
\,
|a_i^{\rm T}-b_i^{\rm T}|
\Bigr),
\qquad
{\rm T}\in\{{\rm TE,TM,TEM}\}.
\]
In particular, as the number of subdomains tends to infinity, the convergence factor of each modal iteration tends to a constant. If this limiting constant is strictly smaller than \(1\) for all modes, then the one-level Schwarz method is weakly scalable.
\end{theorem}

\subsubsection{Transparent limit and nilpotency}

The limiting-spectrum formula of Theorem~\ref{thm:spectre-limit} requires \(a_i^{\rm T}\neq 0\). In the transparent limit, however, one has \(a_i^{\rm T}=0\), so the theorem no longer applies. In that case, the convergence can nevertheless be characterized explicitly: the iteration matrix becomes nilpotent. This occurs, in particular, for an infinite PML, which yields the exact Dirichlet-to-Neumann operator.

\begin{proposition}[Nilpotency in the transparent limit]\label{prop:nilpotent}
Assume that the transmission operator \(\mathcal T\) is the exact Dirichlet-to-Neumann operator, equivalently the limit of a PML with \(\ell\to+\infty\). Then, for every modal family \({\rm T}\in\{{\rm TE,TM,TEM}\}\) and every mode index \(i\),
\[
\lambda_i^{\rm T}=-\eta_i^{\rm T},
\qquad
a_i^{\rm T}=0,
\qquad
b_i^{\rm T}=e^{-\imath \beta_i^{\rm T}(L+\delta)}.
\]
Moreover, the modal iteration matrix satisfies
\[
\bigl(I^{\rm T}(i)\bigr)^N=0,
\]
where \(N\) is the number of subdomains.
\end{proposition}
This result is consistent with the classical interpretation of transparent boundary conditions: when the exact outgoing behavior is imposed at each interface, information propagates across at most one subdomain per Schwarz iteration, so the algorithm converges in at most \(N\) steps.

\begin{proof}
In the transparent limit \(\ell\to+\infty\) with \(\sigma>0\), the PML symbols satisfy
\(
\lambda_i^{\rm T}=-\eta_i^{\rm T}.
\)
Substituting this into \eqref{eq:aT}--\eqref{eq:bT} immediately gives
\[
a_i^{\rm T}=0,
\qquad
b_i^{\rm T}=e^{-\imath \beta_i^{\rm T}(L+\delta)}.
\]

It remains to prove that \(\bigl(I^{\rm T}(i)\bigr)^N=0\). Let \(\mathbf R\in\mathbb C^{2N}\). Since \(a_i^{\rm T}=0\), the block matrices in Proposition~\ref{prop:R-ite} reduce to
\[
K_{\rm T}^+(i)=
\begin{bmatrix}
0&0\\
0&b_i^{\rm T}
\end{bmatrix},
\qquad
K_{\rm T}^-(i)=
\begin{bmatrix}
b_i^{\rm T}&0\\
0&0
\end{bmatrix}.
\]
Hence, in the global iteration matrix \(I^{\rm T}(i)\), each nonzero component is shifted by one subdomain at each iteration: left-going components propagate to the left, right-going components propagate to the right, and no coupling remains between the two directions. More explicitly, for \(l=1,\dots,N-1\),
\[
\bigl(I^{\rm T}(i)\mathbf R\bigr)_{2l+1}
=
b_i^{\rm T}\,\mathbf R_{2l-1},
\qquad
\bigl(I^{\rm T}(i)\mathbf R\bigr)_{2l}
=
b_i^{\rm T}\,\mathbf R_{2l+2},
\]
while the two extreme components 
\(
\bigl(I^{\rm T}(i)\mathbf R\bigr)_1=0\), \(\bigl(I^{\rm T}(i)\mathbf R\bigr)_{2N}=0\) vanish.

Therefore, after each multiplication by \(I^{\rm T}(i)\), one additional pair of boundary components becomes zero. After \(N\) iterations, all components vanish, i.e.
\[
\bigl(I^{\rm T}(i)\bigr)^N\mathbf R=0
\qquad
\forall \mathbf R\in\mathbb C^{2N}.
\]
This proves that \(\bigl(I^{\rm T}(i)\bigr)^N=0\).
\end{proof}
Thus, the generic regime is governed by the limiting spectrum of a block Toeplitz matrix, while the transparent limit corresponds to a more favorable situation in which the Schwarz iteration converges in finitely many steps.

\subsection{Illustration of the limiting spectrum}
\label{subsec:limspec}

Here we illustrate the \emph{limiting spectrum}
prediction of Theorem~\ref{thm:spectre-limit} by computing the spectral radius of the
per–mode iteration matrices \(I^{\mathrm T}(i)\) and comparing it with the theoretical
limit as the number of subdomains increases. These experiments are performed at the
\emph{modal level}: no spatial discretization is involved, and the matrices
\(I^{\mathrm T}(i)\) are constructed directly from the analytical expressions of the
coefficients \(a_i^{\mathrm T}\) and \(b_i^{\mathrm T}\). The goal is therefore to
verify the limiting-spectrum result itself. Only in a later section, dedicated to the numerical results, we report weak and strong scalability results for
fully discretized 3D Maxwell problems with several transmission conditions and overlaps.

In the following we will examine how fast the spectral radius of the per–mode continuous iteration matrix
\(I^{\mathrm T}(i)\) converges to the limiting prediction as the number of subdomains
\(N\) increases. We fix \(k=10\), a core length \(L=1\), and overlap \(\delta=0.1\). {For the illustrations below, we have replaced the discrete index $i$ by a continuous variable $r$, which simply amounts to compute $I^{\mathrm T}(r)$ where in the expression of $I^{\rm T}(i)$ we replaced the eigenvalues $\lambda^\mathcal{D}_i$ and $\lambda^{\mathcal{N}}_i$ by a positive real value $r$ (for TE and TM modes). In that case, according to Corollary \ref{cor:MH-dictionary}, we have $I^{\rm TE}(r) = I^{\rm TM}(r)$. For TEM modes, let us mention that the results are the same as for TE modes taking $r = 0$.}

\paragraph{Impedance transmission}
Figure~\ref{fig3:convFactTE-TM} shows \(\rho\bigl(I^{\rm TE}(r)\bigr)\) (left) and
\(\rho\bigl(I^{\rm TM}(r)\bigr)\) (right) for \(N\in\{5,10,\dots,35\}\) with the impedance
condition \eqref{def:op-T-ik}. The dashed line is the limiting spectrum of
Theorem~\ref{thm:spectre-limit}. Convergence to the limit is rapid for evanescent modes
(\(r\ge k\)), and remains fast for propagative modes. {Let us remark that, as mentioned above, both spectral radii are the same for TE and TM modes. Therefore, for the next results, we only show the results for TE modes.}

\begin{figure}[h]
  \centering
  \includegraphics[height=4.3cm]{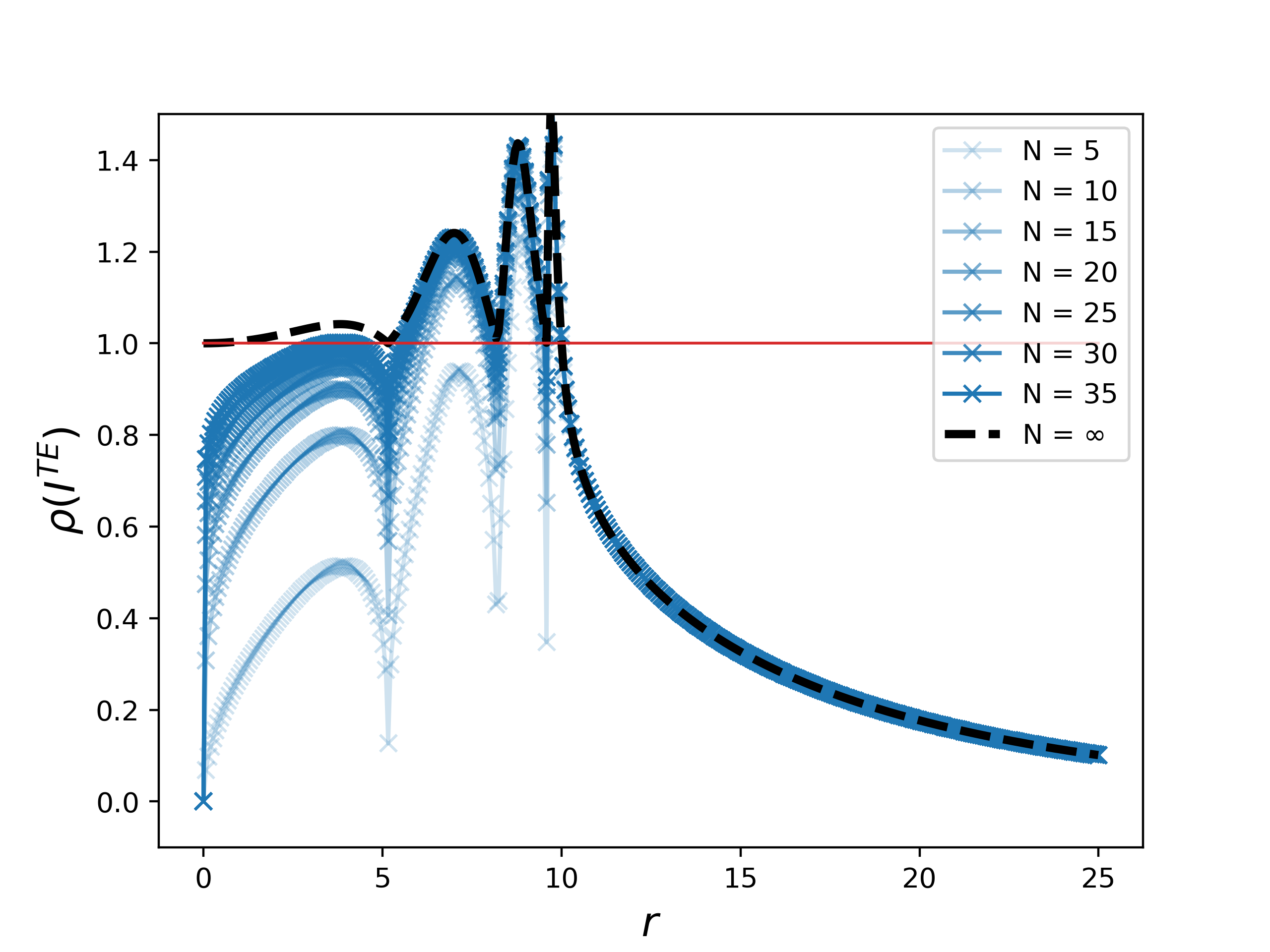}\hfill
  \includegraphics[height=4.3cm]{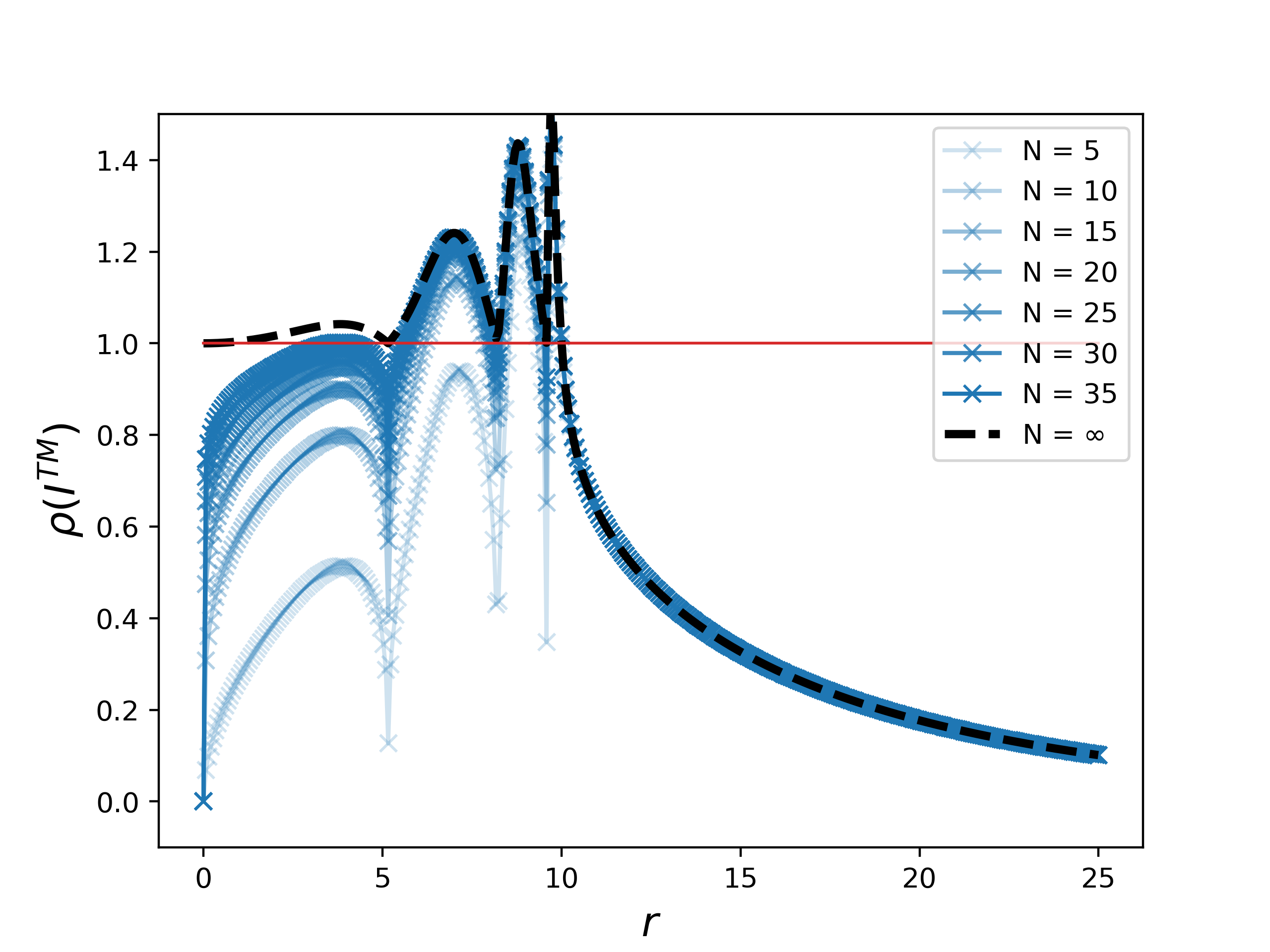}
  \caption{Spectral radius of the iteration matrix for TE (left) and TM (right) versus \(r\ge 0\)
  with impedance transmission \eqref{def:op-T-ik}. Dashed: limiting spectrum.}
  \label{fig3:convFactTE-TM}
\end{figure}

\paragraph{PML transmission}
With a right/left PML of length \(\ell\) and stretch \(\sigma\),
Figure~\ref{fig4:convFactTE-TM-PML} reports \(\rho\bigl(I^{\rm T}(r)\bigr)\) for a ``weak'' PML
(\(\ell=0.1,\ \sigma=5\)) and an almost-transparent PML (\(\ell=1,\ \sigma=10\)). PML yields a
smaller spectral radius than impedance, especially for propagative modes and small \(N\).
When the PML is very strong, convergence to the limiting spectrum is slower for propagative
modes—consistent with the near-nilpotent regime \(a_i^{\mathrm T}\approx 0\), where the limiting
spectrum formula does not apply.

\begin{figure}[h]
  \centering
  \includegraphics[height=4.3cm]{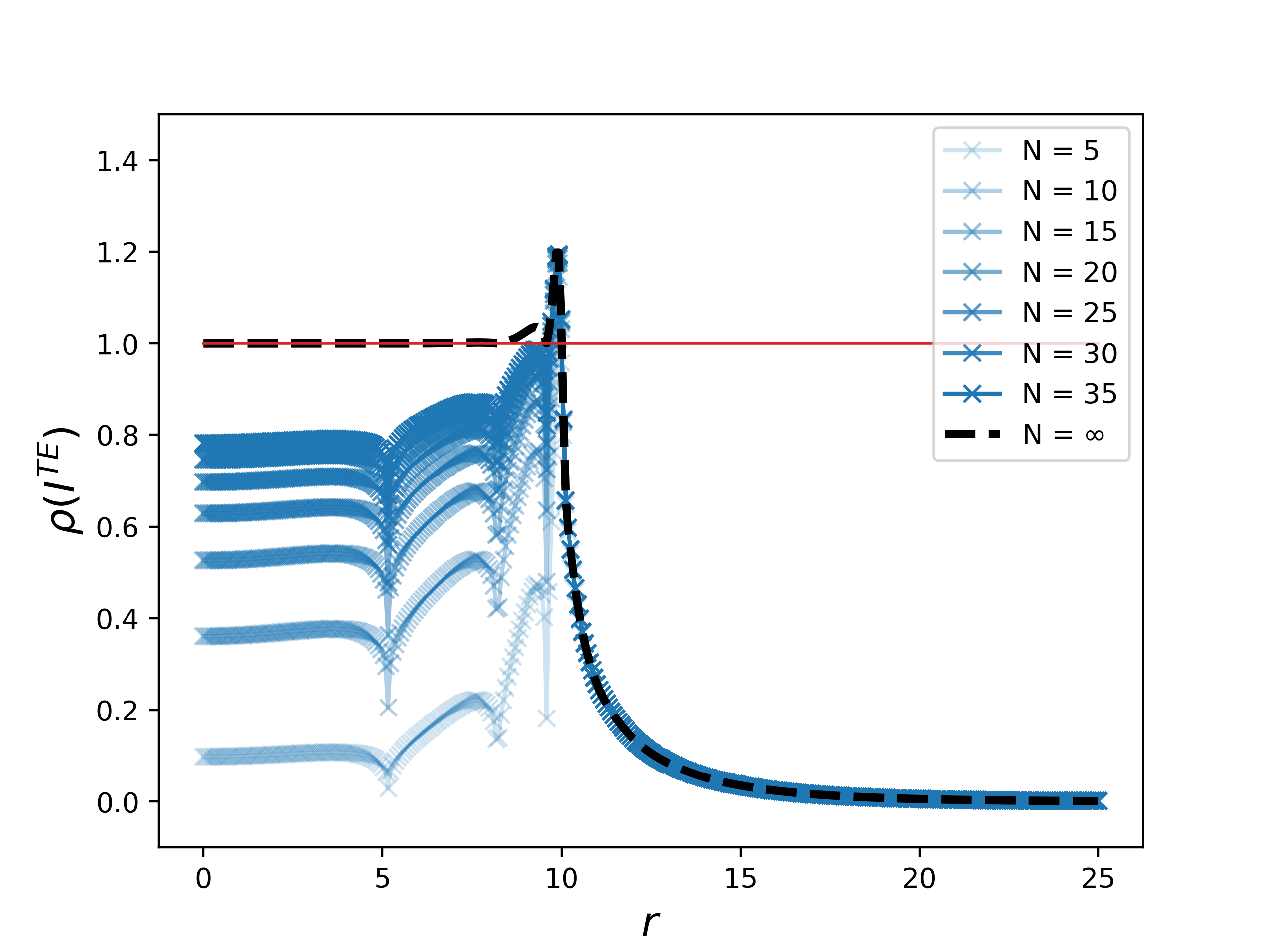}\hfill
  \includegraphics[height=4.3cm]{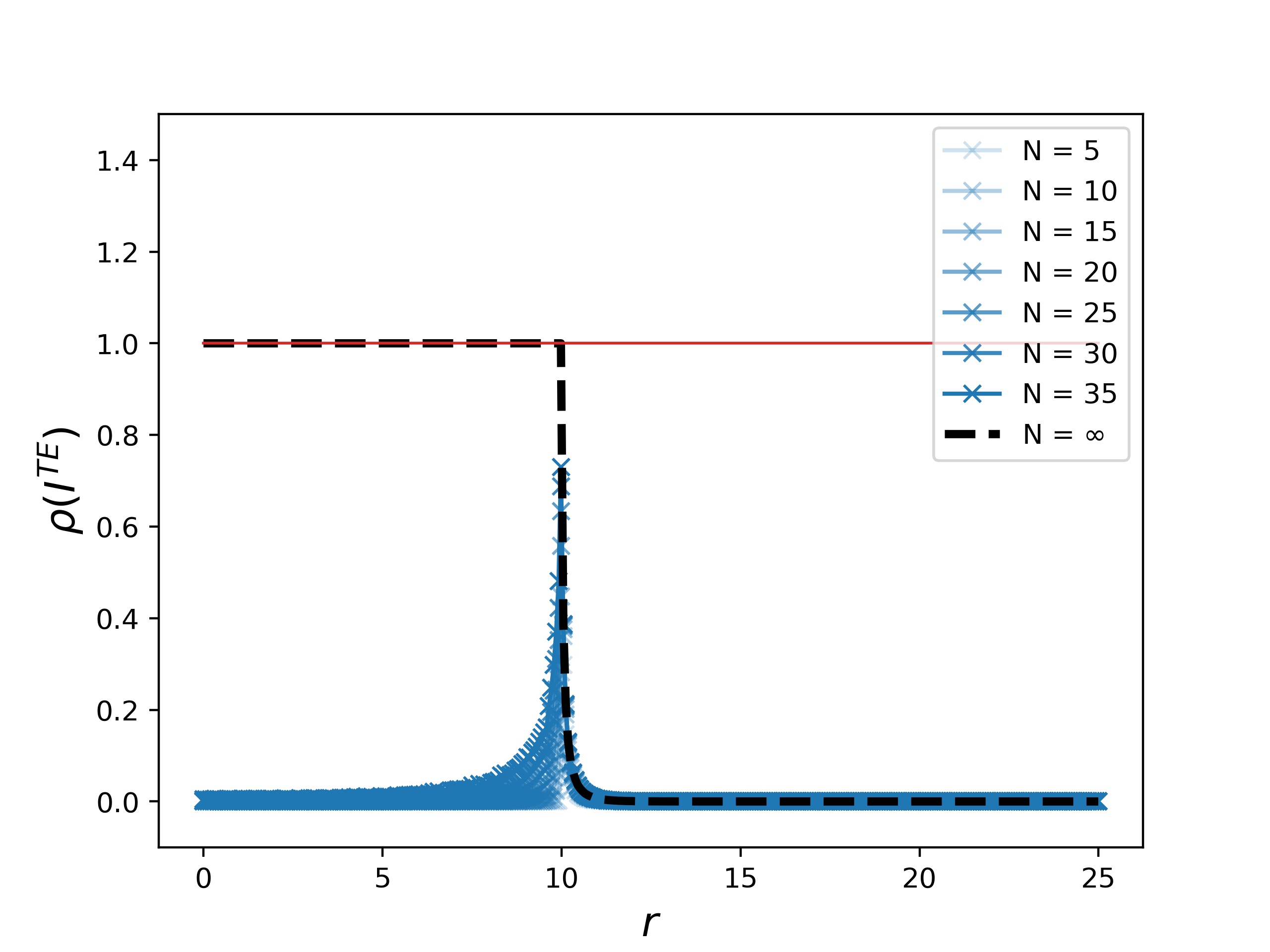}
  \caption{Spectral radius \(\rho\bigl(I^{\rm TE}(r)\bigr)\) with PML transmission.
  Left: \(\ell=0.1,\ \sigma=5\). Right: \(\ell=1,\ \sigma=10\). Dashed: limiting spectrum.}
  \label{fig4:convFactTE-TM-PML}
\end{figure}

\paragraph{Effect of complex frequency (damping)}
We next set \(k=10+\imath\) to model absorptive media (\(\varepsilon''>0\)). As predicted by the
theory, even modest damping drastically reduces \(\rho\bigl(I^{\rm T}(r)\bigr)\); see
Figure~\ref{fig5:convFactTE-kcomplex} for impedance (left) and PML (right). Here the limiting
spectrum is strictly less than one for all \(r\), implying weak scalability.

\begin{figure}[h]
  \centering
  \includegraphics[height=4.3cm]{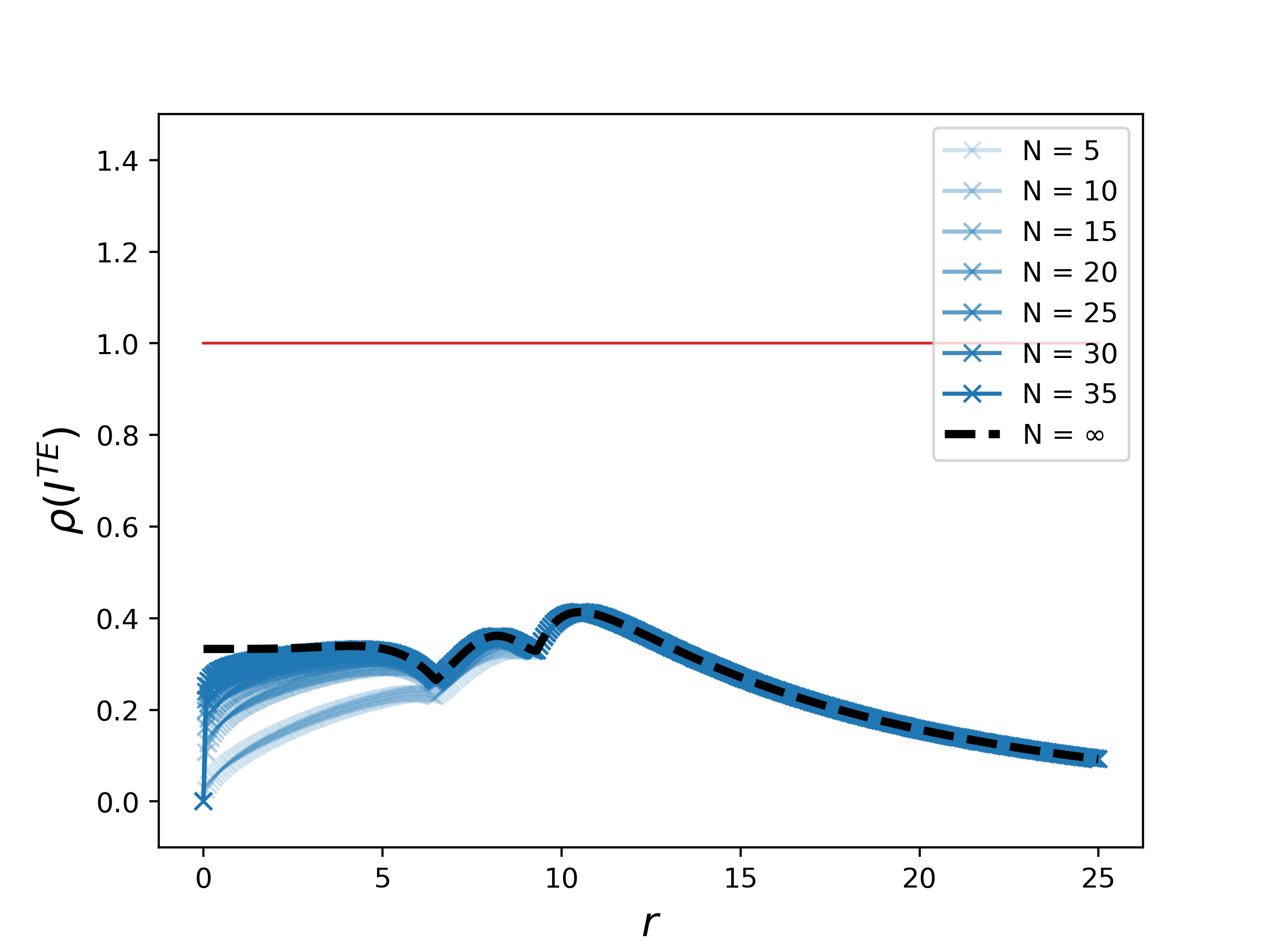}\hfill
  \includegraphics[height=4.3cm]{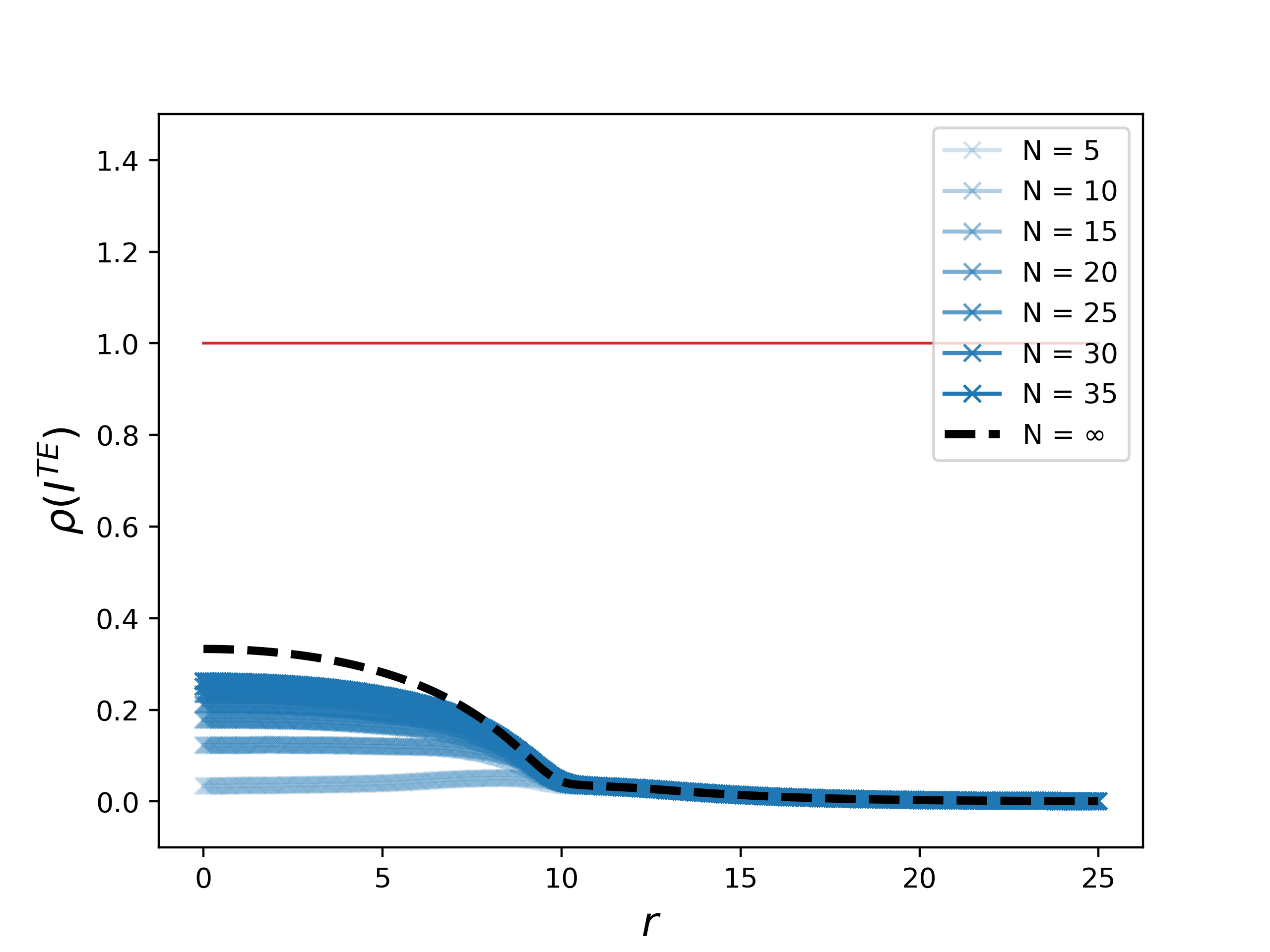}
  \caption{Spectral radius \(\rho\bigl(I^{\rm TE}(r)\bigr)\) with \(k=10+\imath\) using
  impedance (left) and PML (right) with parameters \(\ell=0.1, \sigma=5\). Dashed: limiting spectrum.}
  \label{fig5:convFactTE-kcomplex}
\end{figure}

\paragraph{Takeaways}
\begin{itemize}
  \item Convergence to the limiting spectrum is rapid for evanescent modes and remains fast for
        propagative modes under impedance transmission.
  \item PML transmission reduces the spectral radius overall; very strong PML enters a
        near-nilpotent regime where the limiting-spectrum result is not informative.
  \item Complex frequency (absorption) makes the method weakly scalable: the limiting spectral
        radius is uniformly \(<1\).
\end{itemize}

\paragraph{Interpretation.}
These modal experiments confirm the theoretical analysis. The spectral radius of the
block Toeplitz iteration matrix converges rapidly to the limiting value predicted by
Theorem~\ref{thm:spectre-limit}. Moreover, the behavior of the method is governed by the
modal coefficients \(a_i^{\rm T}\) and \(b_i^{\rm T}\): impedance transmission leads to
a limiting convergence factor close to one for propagative modes, whereas PML
transmission reduces the spectral radius and, in the near-transparent regime
\(a_i^{\rm T}\approx 0\), the iteration becomes nearly nilpotent, explaining the very
fast convergence observed in that case.

\section{Numerical experiments}
\label{sec:numeric}

This section complements the modal analysis of Section~\ref{sec:DD} with fully discretized
three-dimensional experiments. Our goal is twofold. First, we assess how the trends predicted by
the modal stationary analysis---in particular the influence of transmission conditions, overlap,
and damping on weak scalability---carry over to practical solvers. Second, we clarify the relation
between the idealized Schwarz iteration studied in the theory and the optimized restricted additive Schwarz (ORAS)-preconditioned GMRES
solver used in the computations. 

\paragraph{From the stationary Schwarz iteration to ORAS-preconditioned GMRES.}
The analysis in Section~\ref{sec:DD} concerns the \emph{stationary} one-level Schwarz iteration.
At the continuous modal level, each mode is propagated by the matrix \(I^{\rm T}(i)\), and the
spectral radius of this matrix governs the convergence of the pure fixed-point algorithm.

In the fully discretized setting, however, we do not use this stationary iteration as a solver.
Instead, we use one Schwarz sweep as a \emph{preconditioner} inside GMRES. More precisely, let
\(A_h\) be the global finite element matrix associated with the discretized Maxwell problem. The
ORAS preconditioner defines a linear operator \(M_{ORAS}^{-1}\), and the actual solver is
right-preconditioned GMRES applied to
\[
A_h M_{ORAS}^{-1} \mathbf y = \mathbf b,
\qquad
\mathbf x = M_{ORAS}^{-1}\mathbf y.
\]
Thus, the iteration counts reported below measure the quality of the Schwarz preconditioner rather
than the convergence of the stationary Schwarz method itself. The two points of view are nevertheless closely related. The stationary Schwarz iteration for the
error may be written abstractly as
\[
\mathbf e^{m+1} = \bigl(I-M^{-1}A\bigr)\mathbf e^m,
\]
where \(M^{-1}\) denotes one Schwarz sweep. In the modal continuous analysis, this error
propagation operator reduces mode by mode to the matrices \(I^{\rm T}(i)\). In the discrete
setting, ORAS is the finite-dimensional realization of the same idea: it combines local subdomain
solves with the same transmission conditions and overlap pattern, together with a partition of
unity. Therefore, although the theoretical results rigorously concern the stationary iteration,
they remain informative for ORAS-preconditioned GMRES because they predict when one Schwarz sweep
is a good or poor approximation of the inverse operator.

In particular, the theory also explains why the pure one-level stationary algorithm is not used as
a practical solver here. In the weakly damped or undamped regimes, the spectral radius of the modal
iteration matrix is often close to \(1\), and may even exceed \(1\) for some modes, so the
stationary iteration is either very slow or divergent. GMRES is therefore used to stabilize and
accelerate the iteration, while ORAS provides the underlying domain decomposition preconditioning.

\paragraph{Discrete setting.}
We now consider fully discretized problems. The global solver is right-preconditioned GMRES with
relative tolerance \(10^{-5}\), preconditioned by ORAS with overlap \(\delta\) and either impedance or
PML transmission conditions at subdomain interfaces. The mesh uses lowest-order N\'ed\'elec
elements with approximately 10 points per wavelength.

\paragraph{ORAS preconditioner.}
For all experiments we use the ORAS preconditioner
\[
M_{ORAS}^{-1}= \sum_{j=1}^{N} R_j^T D_j A_j^{-1} R_j,
\]
where \(R_j:\Omega\to\Omega_j\) is the restriction operator, \(R_j^T:\Omega_j\to\Omega\) the
prolongation operator, and \(D_j\) is the partition of unity, satisfying
\(
\sum_{j=1}^{N} R_j^T D_j R_j = \text{Id}.
\)
The local matrices \(A_j\) are the discretizations of the subdomain problems
\begin{equation}\label{pb:Schwarz-DD}
	\left\{\begin{array}{lcl}
		\mathcal{L} (\mathbf{E}) = 0 & \text{in} & \Omega_j,\\
		\mathcal{B}(\mathbf{E}) = 0 & \text{on} & \Gamma_{j},\\
		\mathbf{E} \times \mathbf{n} = 0 & \text{on} & \partial \Omega_j \setminus \Gamma_{j}.
	\end{array}\right.
\end{equation}
Hence the local solves entering ORAS are precisely the discrete counterparts of the local Schwarz
problems analyzed in Section~\ref{sec:DD}. (see also \cite{StCyr:2007:OMA} for an explanation on the continuous-discrete link of the Schwarz iterations). 

\paragraph{Overlap and discretization.}
In the finite element setting, the overlap width is specified by the number \(s\) of layers of mesh
elements added to each subdomain. Starting from a non-overlapping partition, an overlap level
\(s\) means that each subdomain is enlarged by \(s\) element layers, so that the geometric overlap
between neighboring subdomains is \(\delta = 2hs\), where \(h\) is the mesh size.

The numerical experiments are performed using the open-source finite element software
FreeFEM~\cite{Hecht:2012:FF}. The implementation uses the \textit{ffddm} framework
\cite{Tournier:2019:FFD}, which provides high-level parallel domain decomposition tools in
FreeFEM.

\paragraph{Geometry, decomposition, and damping parameter.}
Unless otherwise stated, the computational domain is a rectangular waveguide of size
\(1\times 1\times N\), decomposed into \(N\) strip-like subdomains of equal axial length. We
report GMRES iteration counts as functions of the number of subdomains \(N\), for several
real wavenumbers \(k_r := \text{Re}(k)\) and several values of an absorption parameter
\[
k_d := \operatorname{Im}(k)\in\left\{0,\frac{1}{k_r},1,k_r\right\}.
\]
To avoid confusion with the PML stretch parameter $\sigma$ (usually used for the conductivity), we consistently denote the material damping by
\(k_d\). In the experiments, \(k_d\) is introduced as a complex shift in the material and plays the
same role as the complex frequency considered in Section~\ref{subsec:limspec}.

\paragraph{How to read the tables.}
Each table is organized as follows. The rows correspond to the amount of damping \(k_d\), with the
three subrows giving the wavenumbers \(k_r=5,10,15\). The columns correspond to the number of
subdomains \(N\). Large growth of the iteration count with \(N\) indicates poor weak scalability;
nearly constant iteration counts indicate good weak scalability. When \textsf{DNC} appears, GMRES
did not reach the prescribed tolerance within 2000 iterations, which signals that the underlying
one-level preconditioner is too weak in that regime.

\subsection{3D weak scalability}
\label{subsec:weak3d}

We first study weak scalability by increasing the length of the waveguide together with the number
of subdomains, so that each subdomain keeps the same physical size. This is the discrete setting
closest to the modal theory of Section~\ref{sec:DD}. In particular, the rows with \(k_d=0\)
correspond to the most challenging case, where the stationary one-level Schwarz iteration is not
expected to converge robustly; the GMRES iteration counts therefore measure how much the Krylov
acceleration compensates for this lack of contractivity.

\paragraph{Impedance vs.\ PML; overlap \(s=2\)}
Tables~\ref{tab:weak-imp-o2} and~\ref{tab:weak-pml-o2} compare impedance and PML transmission for
overlap \(s=2\). Several observations should be emphasized. First, in the undamped case \(k_d=0\), the iteration count grows strongly with \(N\), especially
for the larger wavenumbers. This is consistent with the modal analysis: without damping, the
one-level stationary Schwarz method is not contractive enough to yield weak scalability, and in
the most difficult regimes the preconditioned GMRES solver also eventually fails to converge within
the iteration cap. Second, as soon as damping is introduced, the iteration counts become much less sensitive to
\(N\). For \(k_d=1\) and \(k_d=k_r\), they are essentially flat, which is the discrete
counterpart of the modal prediction that the limiting convergence factor becomes uniformly smaller
than one. Third, PML transmission improves the preconditioner relative to impedance transmission, most
clearly in the low- and moderately damped regimes. This mirrors the smaller modal convergence
factors observed in Section~\ref{subsec:limspec}.

\begin{table}[h]
\centering
\setlength{\tabcolsep}{3pt}
\renewcommand{\arraystretch}{1.05}
\caption{GMRES iterations (tol \(10^{-5}\)) for a rectangular guide \(1\times 1\times N\), overlap \(2\),
impedance transmission. Initial guess: $\mathrm{TE}_{10}$ mode.}
\label{tab:weak-imp-o2}
\begin{tabular}{|c||c|c|c|c|c|c|c|c|c|}
\hline
$k_d$ & $ k_r \backslash N$ & 5 & 10 & 20 & 40 & 60 & 80 & 100 & 120 \\
\hline
 & 5 & 29 & 64 & 131 & 269 & 412 & 538 & 682 & 798 \\
$0$ & 10 & 64 & 165 & 413 & 960 & 1422 & 1909 & DNC & DNC \\
 & 15 & 46 & 106 & 221 & 442 & 794 & 1163 & 1530 & 1756 \\
\hline\hline
 & 5 & 24 & 41 & 64 & 89 & 94 & 100 & 110 & 112 \\
$\frac{1}{k_r}$ & 10 & 36 & 67 & 114 & 171 & 217 & 253 & 281 & 288 \\
 & 15 & 37 & 76 & 136 & 231 & 274 & 338 & 399 & 462 \\
\hline\hline
 & 5 & 12 & 17 & 20 & 22 & 22 & 22 & 22 & 22 \\
$1$ & 10 & 14 & 20 & 24 & 28 & 28 & 28 & 28 & 28 \\
 & 15 & 14 & 20 & 23 & 29 & 29 & 29 & 29 & 29 \\
\hline\hline
 & 5 & 6 & 6 & 6 & 6 & 7 & 7 & 7 & 7 \\
$k_r$ & 10 & 5 & 5 & 6 & 6 & 6 & 6 & 6 & 6 \\
 & 15 & 5 & 5 & 6 & 6 & 6 & 6 & 6 & 6 \\
\hline
\end{tabular}
\end{table}

\begin{table}[h]
\centering
\setlength{\tabcolsep}{3pt}
\renewcommand{\arraystretch}{1.05}
\caption{GMRES iterations (tol \(10^{-5}\)) for a rectangular guide \(1\times 1\times N\), overlap \(2\),
PML transmission. Initial guess: $\mathrm{TE}_{10}$ mode.}
\label{tab:weak-pml-o2}
\begin{tabular}{|c||c|c|c|c|c|c|c|c|c|}
\hline
$k_d$ & $ k_r \backslash N$ & 5 & 10 & 20 & 40 & 60 & 80 & 100 & 120 \\
\hline
 & 5 & 30 & 65 & 135 & 287 & 427 & 554 & 691 & 805 \\
$0$ & 10 & 58 & 136 & 306 & 682 & 1041 & 1458 & 1688 & DNC \\
 & 15 & 34 & 69 & 136 & 259 & 526 & 714 & 916 & 1088 \\
\hline\hline
 & 5 & 22 & 37 & 57 & 78 & 83 & 87 & 98 & 100 \\
$\frac{1}{k_r}$ & 10 & 32 & 56 & 92 & 144 & 178 & 201 & 229 & 240 \\
 & 15 & 28 & 51 & 87 & 146 & 168 & 210 & 245 & 283 \\
\hline\hline
 & 5 & 11 & 17 & 20 & 22 & 22 & 22 & 22 & 22 \\
$1$ & 10 & 13 & 19 & 23 & 28 & 28 & 28 & 28 & 28 \\
 & 15 & 14 & 21 & 24 & 29 & 29 & 29 & 29 & 29 \\
\hline\hline
 & 5 & 6 & 6 & 6 & 6 & 6 & 6 & 6 & 6 \\
$k_r$ & 10 & 5 & 6 & 6 & 6 & 6 & 6 & 6 & 6 \\
 & 15 & 5 & 6 & 6 & 6 & 6 & 6 & 6 & 6 \\
\hline
\end{tabular}
\end{table}

\paragraph{Effect of overlap}
Table~\ref{tab:weak-imp-o4} reports the same weak-scalability experiment as
Table~\ref{tab:weak-imp-o2}, but with impedance transmission and overlap \(s=4\) instead of
\(s=2\). The effect of increasing the overlap is most visible in the undamped and weakly damped
regimes, where the iteration counts are systematically reduced. For example, at \(k_d=0\) and
\(k_r=10\), increasing the overlap delays the onset of very poor convergence, although it does not
remove it entirely. This is again consistent with the modal picture: larger overlap improves the
quality of one Schwarz sweep, but without damping it does not fundamentally cure the lack of weak
scalability of the stationary one-level method.

\begin{table}[h]
\centering
\setlength{\tabcolsep}{3pt}
\renewcommand{\arraystretch}{1.05}
\caption{GMRES iterations (tol \(10^{-5}\)) for a rectangular guide \(1\times 1\times N\),
impedance transmission, overlap \(4\). Initial guess: $\mathrm{TE}_{10}$ mode.}
\label{tab:weak-imp-o4}
\begin{tabular}{|c||c|c|c|c|c|c|c|c|c|}
\hline
$k_d$ & $ k_r \backslash N$ & 5 & 10 & 20 & 40 & 60 & 80 & 100 & 120 \\
\hline
 & 5 & 30 & 64 & 131 & 261 & 401 & 526 & 666 & 779 \\
$0$ & 10 & 64 & 150 & 327 & 790 & 1195 & 1769 & DNC & DNC \\
 & 15 & 42 & 95 & 193 & 386 & 710 & 1009 & 1375 & 1592 \\
\hline\hline
 & 5 & 24 & 42 & 65 & 88 & 94 & 100 & 109 & 111 \\
$\frac{1}{k_r}$ & 10 & 37 & 67 & 123 & 191 & 234 & 278 & 316 & 329 \\
 & 15 & 37 & 74 & 132 & 213 & 263 & 316 & 371 & 424 \\
\hline\hline
 & 5 & 13 & 18 & 20 & 22 & 22 & 22 & 22 & 22 \\
$1$ & 10 & 13 & 19 & 23 & 27 & 28 & 28 & 28 & 28 \\
 & 15 & 13 & 19 & 23 & 29 & 29 & 29 & 29 & 29 \\
\hline\hline
 & 5 & 6 & 6 & 6 & 6 & 6 & 6 & 6 & 6 \\
$k_r$ & 10 & 4 & 5 & 5 & 5 & 5 & 5 & 5 & 5 \\
 & 15 & 4 & 4 & 4 & 5 & 5 & 5 & 5 & 5 \\
\hline
\end{tabular}
\end{table}

\paragraph{Sensitivity to the initial guess and to the cross-section}
Tables~\ref{tab:weak-imp-rand} and~\ref{tab:weak-imp-cyl} test the robustness of the previous
observations with respect to the initial guess and the waveguide geometry.

Table~\ref{tab:weak-imp-rand} repeats the impedance/overlap-2 experiment with a random initial
guess instead of the \(\mathrm{TE}_{10}\) mode. The same qualitative trends are observed:
undamped cases deteriorate rapidly with \(N\), while damping restores nearly flat iteration
counts. The random initial guess is generally less favorable, which explains the larger counts in
the difficult regimes, but it does not change the overall interpretation.

Table~\ref{tab:weak-imp-cyl} reports analogous results for a cylindrical waveguide of diameter
\(1\). The mesh of the waveguide decomposed into $N = 10$ subdomains is shown in Figure~\ref{fig:cyl-mesh-dd}. The same pattern is observed there as well: poor weak scalability without damping, strong
improvement under damping, and nearly flat counts in the highly absorptive regime. This shows that
the conclusions are not specific to the rectangular guide used in the modal analysis.

\begin{figure}[h]
	\centering
	\includegraphics[height=3.5cm]{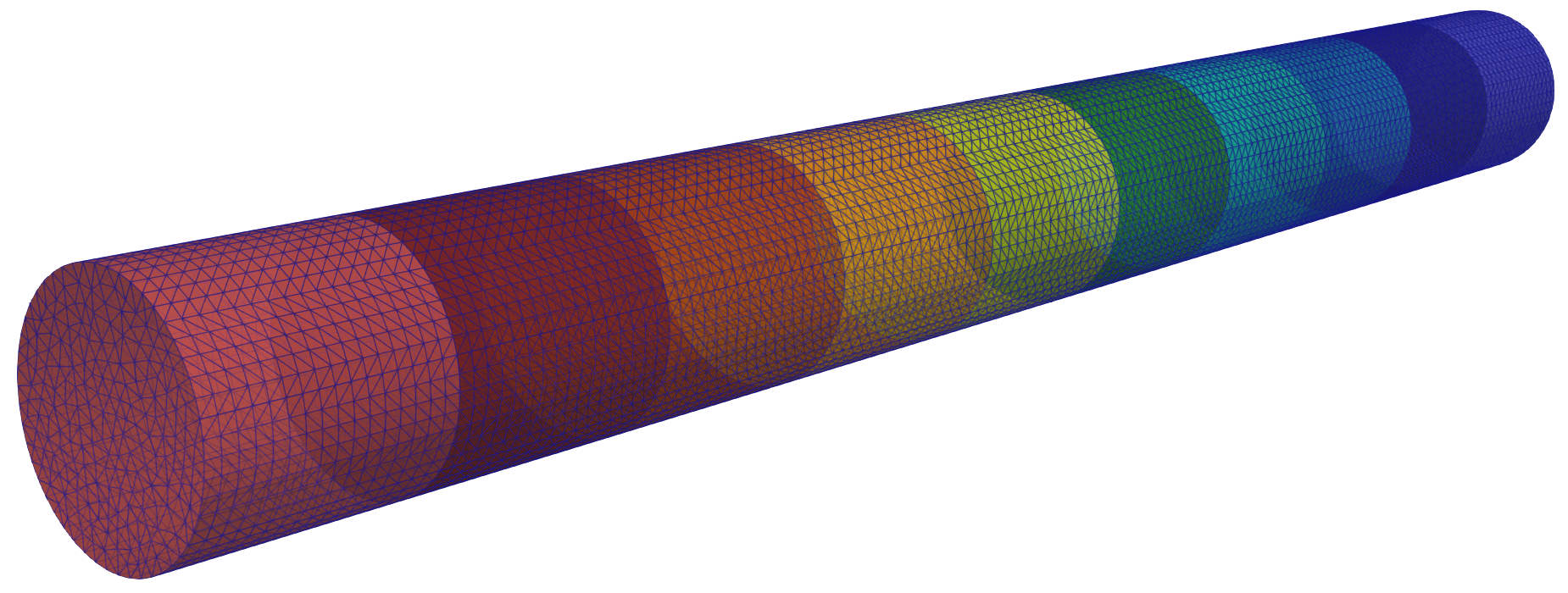}
	\caption{Mesh of the cylindrical waveguide partitioned into $N = 10$ subdomains.}\label{fig:cyl-mesh-dd}
\end{figure}

\begin{table}[h]
\centering
\setlength{\tabcolsep}{3pt}
\renewcommand{\arraystretch}{1.05}
\caption{GMRES iterations (tol \(10^{-5}\)) for a rectangular guide \(1\times 1\times N\),
impedance transmission, overlap \(2\), \emph{random} initial guess.}
\label{tab:weak-imp-rand}
\begin{tabular}{|c||c|c|c|c|c|c|c|c|c|}
\hline
$k_d$ & $ k_r \backslash N$ & 5 & 10 & 20 & 40 & 60 & 80 & 100 & 120 \\
\hline
 & 5 & 33 & 68 & 140 & 286 & 429 & 541 & 692 & 819 \\
$0$ & 10 & 84 & 211 & 514 & 1128 & 1783 & DNC & DNC & DNC \\
 & 15 & 81 & 176 & 368 & 741 & DNC & DNC & DNC & DNC \\
\hline\hline
 & 5 & 28 & 48 & 69 & 93 & 96 & 107 & 114 & 114 \\
$\frac{1}{k_r}$ & 10 & 58 & 101 & 151 & 207 & 246 & 270 & 291 & 304 \\
 & 15 & 67 & 126 & 206 & 311 & 383 & 444 & 500 & 537 \\
\hline\hline
 & 5 & 16 & 19 & 23 & 24 & 24 & 24 & 24 & 24 \\
$1$ & 10 & 20 & 25 & 29 & 31 & 31 & 31 & 31 & 31 \\
 & 15 & 23 & 28 & 32 & 34 & 35 & 35 & 35 & 35 \\
\hline\hline
 & 5 & 7 & 8 & 8 & 8 & 8 & 8 & 8 & 8 \\
$k_r$ & 10 & 7 & 7 & 8 & 8 & 8 & 8 & 8 & 8 \\
 & 15 & 8 & 8 & 8 & 8 & 8 & 8 & 8 & 8 \\
\hline
\end{tabular}
\end{table}

\begin{table}[h]
\centering
\setlength{\tabcolsep}{3pt}
\renewcommand{\arraystretch}{1.05}
\caption{GMRES iterations (tol \(10^{-5}\)) for a cylindrical guide (diameter \(1\), length \(N\)),
impedance transmission, overlap \(2\), random initial guess.}
\label{tab:weak-imp-cyl}
\begin{tabular}{|c||c|c|c|c|c|c|c|c|c|}
\hline
$k_d$ & $ k_r \backslash N$ & 5 & 10 & 20 & 40 & 60 & 80 & 100 & 120 \\
\hline
 & 5 & 30 & 60 & 117 & 228 & 338 & 449 & 560 & 672 \\
$0$ & 10 & 46 & 95 & 187 & 358 & 521 & 686 & 848 & 1024 \\
 & 15 & 91 & 183 & 348 & 683 & 1000 & DNC & DNC & DNC \\
\hline\hline
 & 5 & 26 & 46 & 67 & 87 & 93 & 98 & 102 & 102 \\
$\frac{1}{k_r}$ & 10 & 39 & 72 & 115 & 165 & 199 & 220 & 236 & 248 \\
 & 15 & 74 & 127 & 197 & 302 & 380 & 461 & 530 & 592 \\
\hline\hline
 & 5 & 15 & 18 & 21 & 21 & 21 & 21 & 21 & 21 \\
$1$ & 10 & 19 & 24 & 27 & 29 & 29 & 29 & 29 & 29 \\
 & 15 & 23 & 28 & 31 & 34 & 34 & 34 & 34 & 34 \\
\hline\hline
 & 5 & 7 & 7 & 7 & 7 & 7 & 7 & 7 & 7 \\
$k_r$ & 10 & 7 & 7 & 7 & 8 & 8 & 8 & 8 & 8 \\
 & 15 & 7 & 8 & 8 & 8 & 8 & 8 & 8 & 8 \\
\hline
\end{tabular}
\end{table}

\paragraph{Summary (weak scalability).}
The weak-scalability tables should be interpreted as follows. The one-level stationary Schwarz
iteration analyzed in Section~\ref{sec:DD} is not a robust solver in the undamped regime: its
modal convergence factor is too close to one, and can exceed one for some modes. In the fully
discrete setting, ORAS-preconditioned GMRES is still able to solve many of these problems, but the
iteration counts grow rapidly with the number of subdomains and eventually lead to
\textsf{DNC}. By contrast, absorption and improved transmission conditions make one Schwarz sweep a
much better preconditioner. This is reflected in the nearly flat GMRES iteration counts for \(k_d=1\) and
\(k_d=k_r\), and in the systematic improvement obtained with PML transmission. Overall, the
tables are in good qualitative agreement with the modal theory: damping and more effective
transmission conditions are the mechanisms that restore weak scalability.

\subsection{3D strong scalability (fixed-size domain)}
\label{subsec:strong3d}

We next turn to strong scalability. Here the physical domain is fixed to a rectangular guide of
size \(1\times 1\times 40\), while the number of subdomains \(N\) is increased. This situation is
different from weak scalability: since the physical problem is unchanged, dividing it into more
subdomains introduces more interfaces and therefore tends to degrade the quality of a one-level
Schwarz preconditioner.

Tables~\ref{tab:strong-imp-o0}--\ref{tab:strong-pml-o2} confirm this behavior. In the undamped
case \(k_d=0\), the iteration count grows substantially with \(N\), which is expected for a
one-level method on a fixed domain. Increasing the overlap or replacing impedance transmission by
PML slightly improves the results in most of the cases, but the main stabilization mechanism remains the presence of damping.

More precisely, Table~\ref{tab:strong-imp-o0} gives the baseline impedance results with overlap \(1\). Table~\ref{tab:strong-imp-o2} shows that overlap \(2\) improves the
iteration counts, especially in the less damped regimes. Table~\ref{tab:strong-pml-o2} shows that
PML transmission further improves performance, most clearly for \(k_d=0\) and \(k_d=\frac{1}{k_r}\).
In the strongly damped regime \(k_d=k_r\), all three tables show that the iteration counts
remain very small and only weakly dependent on \(N\).

\begin{table}[h]
\centering
\setlength{\tabcolsep}{3pt}
\renewcommand{\arraystretch}{1.05}
\caption{Strong scalability: rectangular guide \(1\times 1\times 40\), impedance transmission, overlap 1.}
\label{tab:strong-imp-o0}
\begin{tabular}{|c||c|c|c|c|c|c|}
\hline
$k_d$ & $ k_r \backslash N$ & 5 & 10 & 20 & 40 & 80 \\
\hline
 & 5  & 35  & 74  & 138  & 294  & 546  \\
$0$         & 10 & 104 & 250 & 476  & 951  & 1908 \\
         & 15 & 45  & 96  & 187  & 378  & 764  \\
\hline\hline
 & 5  & 12  & 22  & 43   & 85   & 174  \\
$\frac{1}{k_r}$          & 10 & 18  & 35  & 68   & 143  & 282  \\
          & 15 & 21  & 40  & 85   & 182  & 388  \\
\hline\hline
 & 5  & 9   & 9   & 12  & 22  & 43  \\
$1$      & 10 & 11  & 12  & 15  & 28  & 53  \\
      & 15 & 11  & 12  & 16  & 29  & 56  \\
\hline\hline
 & 5  & 6   & 7   & 7  & 8  & 11  \\
$k_r$      & 10 & 7  & 7  & 7  & 7  & 8  \\
      & 15 & 7  & 7  & 7  & 7  & 7  \\
\hline
\end{tabular}
\end{table}

\begin{table}[h]
\centering
\setlength{\tabcolsep}{3pt}
\renewcommand{\arraystretch}{1.05}
\caption{Strong scalability: rectangular guide \(1\times 1\times 40\), impedance transmission, overlap \(2\).}
\label{tab:strong-imp-o2}
\begin{tabular}{|c||c|c|c|c|c|c|}
\hline
$k_d$ & $ k_r \backslash N$ & 5 & 10 & 20 & 40 & 80 \\
\hline
 & 5  & 34  & 74  & 124  & 269  & 504  \\
$0$         & 10 & 93 & 226 & 447  & 960  & 1982 \\
         & 15 & 52  & 112  & 221  & 442   & 897  \\
\hline\hline
 & 5  & 12  & 22  & 43   & 89   & 186  \\
$\frac{1}{k_r}$          & 10 & 18  & 38  & 76   & 171  & 375  \\
          & 15 & 20  & 45  & 105  & 231  & 516  \\
\hline\hline
 & 5  & 7   & 8   & 12  & 22  & 42  \\
$1$      & 10 & 9  & 10  & 15  & 28  & 53  \\
      & 15 & 9  & 10  & 15  & 29  & 57   \\
\hline\hline
 & 5  & 5   & 5   & 5  & 6  & 11  \\
$k_r$      & 10 & 5  & 5  & 5  & 6  & 7  \\
      & 15 & 5  & 5  & 6  & 6  & 6  \\
\hline
\end{tabular}
\end{table}

\begin{table}[h]
\centering
\setlength{\tabcolsep}{3pt}
\renewcommand{\arraystretch}{1.05}
\caption{Strong scalability: rectangular guide \(1\times 1\times 40\), PML transmission, overlap \(2\).}
\label{tab:strong-pml-o2}
\begin{tabular}{|c||c|c|c|c|c|c|}
\hline
$k_d$ & $ k_r \backslash N$ & 5 & 10 & 20 & 40 & 80 \\
\hline
 & 5  & 33  & 73  & 143  & 287  & 503  \\
$0$         & 10 & 69 & 159 & 305  & 682  & 1362 \\
         & 15 & 33  & 68  & 138  & 259   & 496  \\
\hline\hline
 & 5  & 10  & 19  & 39   & 78   & 159  \\
$\frac{1}{k_r}$          & 10 & 17  & 28  & 55   & 144  & 288  \\
          & 15 & 17  & 36  & 70   & 146  & 283  \\
\hline\hline
 & 5  & 6   & 7   & 12  & 22  & 42  \\
$1$      & 10 & 8  & 9  & 15  & 28  & 53  \\
      & 15 & 8  & 9  & 15  & 29  & 56   \\
\hline\hline
 & 5  & 5   & 5   & 5  & 6  & 11  \\
$k_r$      & 10 & 5  & 6  & 6  & 6  & 7  \\
      & 15 & 5  & 6  & 6  & 6  & 6  \\
\hline
\end{tabular}
\end{table}

\paragraph{Summary (strong scalability).}
For a fixed physical domain, increasing the number of subdomains deteriorates the performance of a
one-level Schwarz preconditioner in the low- or zero-damping regime, as expected. The role of
GMRES here is again to accelerate a preconditioner that would be ineffective as a pure stationary
iteration. Overlap and PML transmission improve the situation, but the strongest effect comes from
absorption, which keeps the iteration counts moderate and nearly independent of the decomposition.
These trends are fully consistent with the modal analysis of
Section~\ref{subsec:limspec}: damping reduces the difficult propagative effects, while better
transmission conditions improve the quality of the Schwarz sweep used inside ORAS.

\paragraph{Overall interpretation.}
The numerical results should therefore be read in the following way. The theory developed in
Section~\ref{sec:DD} rigorously characterizes the stationary one-level Schwarz method through its
modal iteration matrices. This stationary method is not intended as a robust standalone solver in
the undamped regime, and the tables clearly reflect this through rapidly growing iteration counts
and non-convergence in the most difficult cases. In practice, one uses the same local Schwarz
solves inside the ORAS preconditioner and then relies on GMRES to stabilize and accelerate the
global iteration. The success or failure of ORAS-preconditioned GMRES is therefore not identical to
the convergence of the stationary Schwarz method, but the same modal mechanisms remain visible:
damping, overlap, and more effective transmission conditions all improve the quality of one Schwarz
sweep and therefore the performance of the preconditioner.

\bibliographystyle{siamplain}
\bibliography{paper.bib}
\end{document}